\documentclass[runningheads,svgnames]{llncs}
\usepackage[T1]{fontenc}
\usepackage{graphicx}
\usepackage{amsmath}
\usepackage{amssymb}
\usepackage{amsfonts}
\usepackage{tikz}
\usepackage{float}
\usepackage{hyperref}
\usepackage{cleveref}
\usepackage{tikz-network}
\usetikzlibrary{positioning, quotes}
\usepackage{comment}
\usepackage{orcidlink}
\usepackage{marvosym}
\usepackage{array}

\spnewtheorem*{notation}{Notation}{\bfseries}{\upshape}

\begin{document}

\title{On Eternal Connected Vertex Cover}

\author{Rajat Adak\orcidlink{0009-0001-2723-5550} \and Saraswati Girish Nanoti\orcidlink{0009-0009-7789-8895}}
\authorrunning{R. Adak\and S.G. Nanoti}
\institute{Department of Computer Science and Automation, Indian Institute of Science, Bengaluru, India\\
\email{rajatadak@iisc.ac.in, saraswatig@iisc.ac.in}
} 

\maketitle
\begin{abstract}
In this game played on a connected $n$-vertex graph $G$ ($n\geq 2$), there are two players namely the attacker and the defender. The defender chooses a set of $k$ vertices which induce a connected subgraph. In each move, the attacker attacks an edge. The defender must move at least one guard across this edge. Every other guard may stay put or move to a neighboring vertex, i.e., the defender can move any subset of the remaining $k-1$ guards to their neighboring vertices. It is required that the occupied
vertices induce a connected graph. After this, the attacker attacks again and so on. If the attacker attacks an edge $uv$ such that both $u$ and $v$ do not have a guard, the attacker wins. If the set of vertices occupied by guards do not induce a connected subgraph after the movement is completed, the attacker wins. Therefore, the set of vertices occupied by guards must remain a connected vertex cover after every response. Also, at most one guard must occupy each vertex. The
least number of guards with which the defender can defend every infinite sequence of attacks is denoted
by $ecvc(G)$. The minimum size of a connected vertex cover is denoted by $cvc(G)$. It is known that $cvc(G)\leq ecvc(G)\leq cvc(G)+1$. A necessary condition for $ecvc(G)=cvc(G)$ is that every vertex belongs to some minimum connected vertex cover. We show that this condition is not sufficient: there exists a $32$-vertex graph $G$ that has $cvc(G)=19$ and $ecvc(G)=20$, although every vertex of $G$ belongs to some minimum connected vertex cover.

For connected graphs with minimum degree at least two, we establish the sharp bound $cvc(G)\geq 2|V(G)|-|E(G)|-1$ and $\mathcal F$ denotes the class attaining equality. We prove that $G\in\mathcal F$ if and only if the vertices of degree at least three induce a forest. Within $\mathcal F$, the conditions $ecvc(G)=cvc(G)$, membership of every vertex in some minimum connected vertex cover, and the presence of at least two degree-two vertices on every cycle are equivalent. As applications, we obtain $ecvc(G)=cvc(G)$ for minimally $2$-connected graphs, and we recover the known equality for full subdivisions of connected graphs of minimum degree at least two. In both the families, every minimum connected vertex cover is an eternally winning configuration. This is not true in general for graphs outside $\mathcal{F}$, we show one example of such a graph.
\end{abstract}
\keywords{Eternal connected vertex cover \and Eternal vertex cover
\and Connected vertex cover \and Full subdivisions
\and Minimally $2$-connected graphs}

\section{Introduction}
Let $G$ be a finite simple connected graph with at least one edge. A
\emph{vertex cover} meets every edge of $G$; it is a \emph{connected vertex
cover} if it also induces a connected subgraph. We denote the
minimum size of a connected vertex cover by $cvc(G)$. In the eternal vertex cover (EVC)
game of Klostermeyer and Mynhardt~\cite{klostermeyer2009edge}, guards must
maintain a vertex cover against an arbitrary sequence of edge attacks.
Equality between the eternal vertex cover number and the ordinary vertex
cover number has since been studied from a structural viewpoint
\cite{babu2022graphs,misra2025characterization}, alongside algorithms for particular
graph classes~\cite{paul2023some}, etc.

The \emph{eternal connected vertex cover} (ECVC) problem was introduced by
Fujito and Nakamura~\cite{fujito2020eternal}. In this game played on a connected $n$-vertex graph $G$ ($n\geq 2$), there are two players namely the attacker and the defender. The defender chooses a set of $k$ vertices which induce a connected subgraph. In each move, the attacker attacks an edge. The defender must move at least one guard across this edge. Every other guard may stay put or move to a neighboring vertex, i.e., the defender can move any subset of the remaining $k-1$ guards to their neighboring vertices. It is required that the occupied
vertices induce a connected graph. After this, the attacker attacks again and so on. If the attacker attacks an edge $uv$ such that both $u$ and $v$ do not have a guard, the attacker wins. If the guards do not induce a connected subgraph after the movement is completed, the attacker wins. Therefore, the set of vertices occupied by guards must remain a connected vertex cover after every response. Also, at most one guard must occupy each vertex. The
least number of guards with which the defender can defend every infinite sequence of attacks is denoted
by $ecvc(G)$. Fujito and Nakamura established the bounds
\[
    cvc(G)\leq ecvc(G)\leq cvc(G)+1
\]
and studied ECVC on chordal graphs and several classes defined by their
blocks. The upper bound also holds for $evc(G)$, shown by Klostermeyer and Mynhardt in \cite{klostermeyer2009edge}.  We prove these bounds here for the sake of completeness.

\begin{lemma}[\cite{fujito2020eternal}]
    For any connected graph $G$ with at least two vertices, we have $cvc(G)\leq ecvc(G)\leq cvc(G)+1$.
\end{lemma}

\begin{proof}
    The set of vertices occupied by the guards in the graph $G$ must be a connected vertex cover of $G$ as seen above. Hence, the number of guards required for the defender to win cannot be less than the size of the smallest connected vertex cover of $G$. Therefore, $ecvc(G)\geq cvc(G)$. Now, consider a configuration where there are guards on each vertex of a minimum-sized connected vertex cover $S$ of $G$, and one guard on a vertex $v\notin S$. Since $S$ is a vertex cover, there is no edge of $G$ with both its endpoints outside $S$. If the attacker attacks an edge $xy$ such that both $x$ and $y$ have a guard, the guards on $x$ and $y$ exchange their position, and all the other guards stay on their position. Thus, there is no change in the configuration. Suppose that an edge $uw$ such that $u\in S$ and $w\notin S\cup\{v\}$ is attacked. Since $S$ is a connected vertex cover of $G$, and $G$ is connected, $v$ has a neighbor $z\in S$, and there exists a path $v_1(=z),v_2\ldots v_\ell(=u)$ such that $v_i\in S\,\forall i\in [1,\ell]$. The guard on $u$ moves to $w$, the guard on $v_i$ moves to $v_{i+1}$ for each $i\in [1,\ell-1]$, and the guard on $v$ moves to $z$. Thus, each vertex in $S$ has a guard and there is one guard outside $S$ (on $w$ instead of $v$). Thus, the defender can use this strategy to defend an infinite sequence of attacks on $G$, hence $ecvc(G)\leq cvc(G)+1$.
\qed\end{proof}

More recent work gives complexity results and algorithms for chain
graphs, cographs and other graph classes~\cite{paul2025eternal}. The tight
one-guard gap makes it natural to ask which graphs attain the lower bound.

If $ecvc(G)=cvc(G)$, every vertex belongs to some minimum connected vertex
cover. Indeed, start from a winning configuration with $cvc(G)$ guards. If a
vertex $v$ is unoccupied, an attack on an edge incident with $v$ forces a
guard onto $v$, and the resulting configuration is again a minimum connected
vertex cover. It was asked in \cite{fujito2020eternal} and \cite{paul2025eternal} whether this necessary condition is sufficient. Our first result shows that this necessary condition is not
sufficient: we construct a $32$-vertex graph $G$ in which every vertex
belongs to a minimum connected vertex cover, but $cvc(G)=19$ and
$ecvc(G)=20$. We also examine a one-edge modification $G^\prime$ to study
winning configurations among its minimum connected vertex covers.

Also, observe that in a graph with more than one edge and a degree one vertex, the necessary condition cannot hold, because the degree one vertex cannot belong to a minimum-sized connected vertex cover. Hence, we focus on graphs with minimum degree at least two.

We then identify a class in which the vertex-membership condition does
characterize equality. For every connected graph with minimum degree at
least two, we prove the sharp bound
\[
    cvc(G)\geq 2|V(G)|-|E(G)|-1.
\]
We denote $\mathcal F$ as the class that consists of the graphs attaining equality. We show that a graph belongs
to $\mathcal F$ exactly when its vertices of degree at least three induce
a forest. Within $\mathcal F$, the conditions $ecvc(G)=cvc(G)$, membership
of every vertex in some minimum connected vertex cover, and the presence of
at least two degree-two vertices on every cycle are equivalent. In this
case, every minimum connected vertex cover is an eternally winning
configuration. The criterion gives an exact formula for minimally
$2$-connected graphs, and it recovers the known formula for full
subdivisions of connected graphs of minimum degree at least
two~\cite{paul2025eternal,escoffier2010complexity}.

Section~\ref{sec:Counter-example} constructs the counterexample, and
Section~\ref{sec:modified-example} discusses its one-edge modification.
Section~\ref{sec:structural-criterion} proves the structural criterion and
its recognition form. Section~\ref{sec:applications} gives the two
applications and revisits the counterexample after subdivision.
\section{Insufficiency of the vertex-membership condition}\label{sec:Counter-example}
Here we describe a graph $G(V,E)$ such that every vertex of $G$ belongs to some minimum-sized connected vertex cover of $G$, but $ecvc(G)\neq cvc(G)$. There are $6$ \emph{support vertices} divided into two sets: $A=\{a_1,a_2,a_3\}$ and $C=\{c_1,c_2,c_3\}$. There are $12$ \emph{private} vertices, each support vertex is adjacent to two of these private vertices, which we call as its \emph{private} neighbors. The set of \emph{private} vertices is denoted by $P$. The two private neighbors of the same support vertex are also adjacent to each other. We call the triangle formed by a support vertex and its two private neighbors as a \emph{private triangle}. There are $6$ more vertices divided into two sets $Q:=\{q_1,q_2,q_3\}$ and $R:=\{r_1,r_2,r_3\}$. For $i\in [1,3]$, each $a_i$ and $c_i$ are both adjacent to each $q_i$ and $r_i$. That is, each $G[a_i,q_i,c_i,r_i]$ is a $4$-cycle with $a_i$ and $c_i$ being non-adjacent. There are $8$ more vertices divided into two sets $Z:=\{z_0, z_{12}, z_{23}, z_{13}\}$ and $U:=\{u_0,u_{12}, u_{23}, u_{13}\}$. For each $i<j$, and $i,j\in [1,3]$, each $z_{ij}$ is adjacent to $a_i$ and $a_j$, and each $u_{ij}$ is adjacent to $c_i$ and $c_j$. The vertex $z_0$ is adjacent to $q_1,q_2$ and $a_1$. Similarly, the vertex $u_0$ is adjacent to $r_1,r_2$ and $c_1$. Also, for each $i<j$, and $i,j\in [1,3]$, each $z_{ij}$ is adjacent to $q_i$ and $q_j$, and each $u_{ij}$ is adjacent to $r_i$ and $r_j$. The description of the graph is now complete. It can be seen that $V=A\cup C \cup P \cup Q \cup R \cup U \cup Z$ and $|V|=32$. We also sometimes denote the vertices in a set $S$ as $S$-vertices, for ease of notation, for example, the vertices in $A$ can be denoted as $A$-vertices.

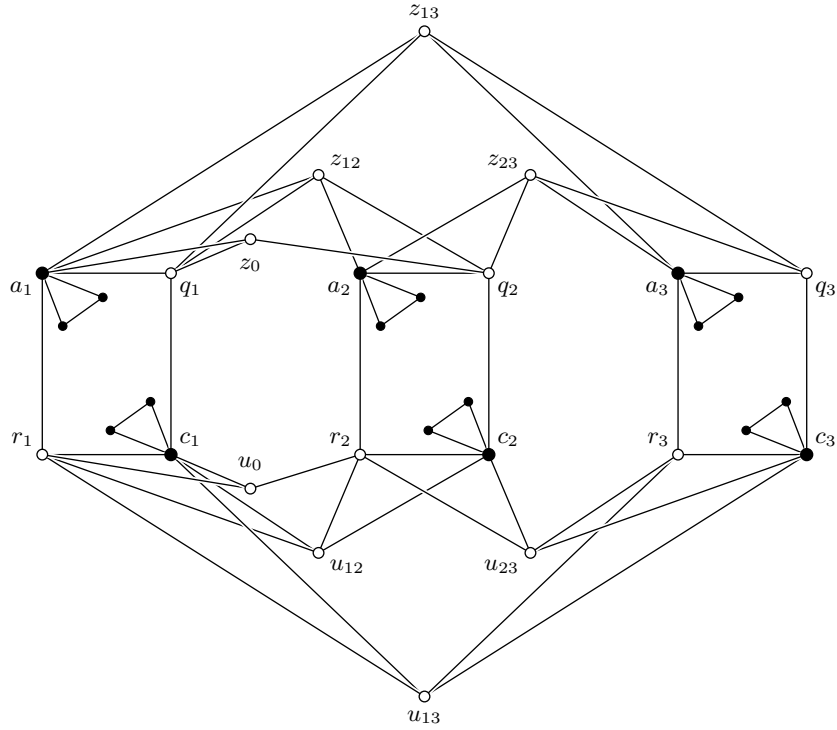
\begin{figure}[tb]
\centering
\begin{tikzpicture}[x=1cm,y=1cm,
  line cap=round,line join=round,
  every node/.style={font=\small},
  edge/.style={draw=black,line width=.5pt,preaction={draw=white,line width=1.8pt}},
  vertex/.style={circle,draw=black,fill=white,line width=.5pt,inner sep=0pt,minimum size=4pt},
  support/.style={vertex,fill=black,minimum size=4.5pt},
  private/.style={circle,fill=black,draw=black,inner sep=0pt,minimum size=3.2pt},
  lab/.style={inner sep=1pt}
]
\coordinate (a1) at (-0.850,1.200);
\coordinate (q1) at (0.850,1.200);
\coordinate (c1) at (0.850,-1.200);
\coordinate (r1) at (-0.850,-1.200);
\coordinate (pa11) at (-0.580,0.500);
\coordinate (pa12) at (-0.050,0.880);
\coordinate (pc11) at (0.580,-0.500);
\coordinate (pc12) at (0.050,-0.880);
\coordinate (a2) at (3.350,1.200);
\coordinate (q2) at (5.050,1.200);
\coordinate (c2) at (5.050,-1.200);
\coordinate (r2) at (3.350,-1.200);
\coordinate (pa21) at (3.620,0.500);
\coordinate (pa22) at (4.150,0.880);
\coordinate (pc21) at (4.780,-0.500);
\coordinate (pc22) at (4.250,-0.880);
\coordinate (a3) at (7.550,1.200);
\coordinate (q3) at (9.250,1.200);
\coordinate (c3) at (9.250,-1.200);
\coordinate (r3) at (7.550,-1.200);
\coordinate (pa31) at (7.820,0.500);
\coordinate (pa32) at (8.350,0.880);
\coordinate (pc31) at (8.980,-0.500);
\coordinate (pc32) at (8.450,-0.880);
\coordinate (z12) at (2.800,2.500);
\coordinate (z23) at (5.600,2.500);
\coordinate (z13) at (4.200,4.400);
\coordinate (z0) at (1.900,1.650);
\coordinate (u12) at (2.800,-2.500);
\coordinate (u23) at (5.600,-2.500);
\coordinate (u13) at (4.200,-4.400);
\coordinate (u0) at (1.900,-1.650);
\draw[edge] (a1) -- (pa11);
\draw[edge] (a1) -- (pa12);
\draw[edge] (pa11) -- (pa12);
\draw[edge] (c1) -- (pc11);
\draw[edge] (c1) -- (pc12);
\draw[edge] (pc11) -- (pc12);
\draw[edge] (a2) -- (pa21);
\draw[edge] (a2) -- (pa22);
\draw[edge] (pa21) -- (pa22);
\draw[edge] (c2) -- (pc21);
\draw[edge] (c2) -- (pc22);
\draw[edge] (pc21) -- (pc22);
\draw[edge] (a3) -- (pa31);
\draw[edge] (a3) -- (pa32);
\draw[edge] (pa31) -- (pa32);
\draw[edge] (c3) -- (pc31);
\draw[edge] (c3) -- (pc32);
\draw[edge] (pc31) -- (pc32);
\draw[edge] (a1) -- (q1);
\draw[edge] (q1) -- (c1);
\draw[edge] (c1) -- (r1);
\draw[edge] (r1) -- (a1);
\draw[edge] (a2) -- (q2);
\draw[edge] (q2) -- (c2);
\draw[edge] (c2) -- (r2);
\draw[edge] (r2) -- (a2);
\draw[edge] (a3) -- (q3);
\draw[edge] (q3) -- (c3);
\draw[edge] (c3) -- (r3);
\draw[edge] (r3) -- (a3);
\draw[edge] (z12) -- (a1);
\draw[edge] (z12) -- (a2);
\draw[edge] (z12) -- (q1);
\draw[edge] (z12) -- (q2);
\draw[edge] (z13) -- (a1);
\draw[edge] (z13) -- (a3);
\draw[edge] (z13) -- (q1);
\draw[edge] (z13) -- (q3);
\draw[edge] (z23) -- (a2);
\draw[edge] (z23) -- (a3);
\draw[edge] (z23) -- (q2);
\draw[edge] (z23) -- (q3);
\draw[edge] (u12) -- (c1);
\draw[edge] (u12) -- (c2);
\draw[edge] (u12) -- (r1);
\draw[edge] (u12) -- (r2);
\draw[edge] (u13) -- (c1);
\draw[edge] (u13) -- (c3);
\draw[edge] (u13) -- (r1);
\draw[edge] (u13) -- (r3);
\draw[edge] (u23) -- (c2);
\draw[edge] (u23) -- (c3);
\draw[edge] (u23) -- (r2);
\draw[edge] (u23) -- (r3);
\draw[edge] (z0) -- (q1);
\draw[edge] (z0) -- (q2);
\draw[edge] (z0) -- (a1);
\draw[edge] (u0) -- (r1);
\draw[edge] (u0) -- (r2);
\draw[edge] (u0) -- (c1);
\node[support] at (a1) {};
\node[vertex] at (q1) {};
\node[support] at (c1) {};
\node[vertex] at (r1) {};
\node[private] at (pa11) {};
\node[private] at (pa12) {};
\node[private] at (pc11) {};
\node[private] at (pc12) {};
\node[support] at (a2) {};
\node[vertex] at (q2) {};
\node[support] at (c2) {};
\node[vertex] at (r2) {};
\node[private] at (pa21) {};
\node[private] at (pa22) {};
\node[private] at (pc21) {};
\node[private] at (pc22) {};
\node[support] at (a3) {};
\node[vertex] at (q3) {};
\node[support] at (c3) {};
\node[vertex] at (r3) {};
\node[private] at (pa31) {};
\node[private] at (pa32) {};
\node[private] at (pc31) {};
\node[private] at (pc32) {};
\node[vertex] at (z12) {};
\node[vertex] at (z23) {};
\node[vertex] at (z13) {};
\node[vertex] at (z0) {};
\node[vertex] at (u12) {};
\node[vertex] at (u23) {};
\node[vertex] at (u13) {};
\node[vertex] at (u0) {};
\node[lab,anchor=north east,xshift=-2pt,yshift=-2pt] at (a1) {$a_{1}$};
\node[lab,anchor=north west,xshift=2pt,yshift=-2pt] at (q1) {$q_{1}$};
\node[lab,anchor=south west,xshift=2pt,yshift=2pt] at (c1) {$c_{1}$};
\node[lab,anchor=south east,xshift=-2pt,yshift=2pt] at (r1) {$r_{1}$};
\node[lab,anchor=north east,xshift=-2pt,yshift=-2pt] at (a2) {$a_{2}$};
\node[lab,anchor=north west,xshift=2pt,yshift=-2pt] at (q2) {$q_{2}$};
\node[lab,anchor=south west,xshift=2pt,yshift=2pt] at (c2) {$c_{2}$};
\node[lab,anchor=south east,xshift=-2pt,yshift=2pt] at (r2) {$r_{2}$};
\node[lab,anchor=north east,xshift=-2pt,yshift=-2pt] at (a3) {$a_{3}$};
\node[lab,anchor=north west,xshift=2pt,yshift=-2pt] at (q3) {$q_{3}$};
\node[lab,anchor=south west,xshift=2pt,yshift=2pt] at (c3) {$c_{3}$};
\node[lab,anchor=south east,xshift=-2pt,yshift=2pt] at (r3) {$r_{3}$};
\node[lab,anchor=south west,xshift=3pt,yshift=1.5pt] at (z12) {$z_{12}$};
\node[lab,anchor=south east,xshift=-3pt,yshift=1.5pt] at (z23) {$z_{23}$};
\node[lab,anchor=south,yshift=4pt] at (z13) {$z_{13}$};
\node[lab,anchor=north,yshift=-6pt] at (z0) {$z_{0}$};
\node[lab,anchor=north west,xshift=3pt,yshift=-1.5pt] at (u12) {$u_{12}$};
\node[lab,anchor=north east,xshift=-3pt,yshift=-1.5pt] at (u23) {$u_{23}$};
\node[lab,anchor=north,yshift=-4pt] at (u13) {$u_{13}$};
\node[lab,anchor=south,yshift=6pt] at (u0) {$u_{0}$};
\end{tikzpicture}
\caption{The counterexample graph $G$. The unlabeled black vertices are the private vertices.}
\label{fig:counterexample}
\end{figure}

\begin{lemma}\label{lem:triangle}
    Any connected vertex cover of $G$ must contain all the vertices in $A\cup C$, and any minimum-sized connected vertex cover of $G$ must contain exactly two vertices from each private triangle: one support vertex and one private vertex.
\end{lemma}

\begin{proof}
    Any vertex in $A\cup C$ is a cut vertex, because deleting it causes its two private neighbors to be disconnected from the rest of the graph, and as there is an edge between these two private vertices, at least one of them must be present in a vertex cover of $G$. Thus, any vertex cover not containing a vertex from $A\cup C$ cannot be connected. 
    
    Also, to cover all the edges of a private triangle, at least two vertices of each private triangle must be present in any vertex cover of $G$. Also, if any connected vertex cover $S$ of $G$ contains all three vertices of a private triangle, then removing a private vertex of this triangle from $S$ still results in a connected vertex cover of $G$. This is because the edges of this private triangle are still covered by the support vertex and the private vertex and all other edges covered by vertices of $S$ still remain covered. Also, any path between two vertices in $S$ (other than the remaining private vertex) is still preserved, as the path either contains no vertex from this private triangle, or contains only the support vertex. The private vertex which has not been removed from $S$ is adjacent to the corresponding support vertex. Thus, the resulting vertex cover is still connected and of size strictly less than $S$, therefore, $S$ is not a minimum-sized connected vertex cover of $G$. Thus, we have shown that any minimum-sized connected vertex cover of $G$ must contain exactly two vertices from each private triangle: one support vertex and one private vertex.
    \qed
\end{proof}

\begin{lemma}\label{lem:minthree}
    Any vertex cover of $G$ must contain at least three vertices from $Q \cup Z$ and at least three vertices from $R \cup U$.
\end{lemma}
\begin{proof}
   The graph $H:=G[Q \cup Z]$ is bipartite with the bipartition $Q \uplus Z$. The graph $H$ also has a matching $M=\{q_1z_{12}, q_2z_{23}, q_3z_{13}\}$ of size $3$. Therefore, any vertex cover of $H$ must contain at least three vertices. Since a vertex cover of $G$ must also cover all the edges of $H$, it must contain at least three vertices from $V(H)$, i.e., $Q \cup Z$. Similarly, any vertex cover of $G$ must contain at least three vertices from $R\cup U$.
   \qed
\end{proof}

\begin{lemma}\label{lem:sizethree}
    The only vertex cover of $G[Q\cup Z]$ of size three is $Q$, and the only vertex cover of $G[R\cup U]$ of size three is $R$.
\end{lemma}

\begin{proof}
    In the proof of \Cref{lem:minthree}, it is shown that $H=G[Q\cup Z]$ is bipartite with bipartition $Q\uplus Z$. Hence, it is clear that $Q$ is a vertex cover of $H$ of size three. Suppose that there exists a vertex cover $S$ of $H$ of size three, such that $S\neq Q$. 
    \begin{itemize}
        \item \textbf{Case 1: $q_1\notin S$.} 

        The vertices $z_{12}$ and $z_{13}$ must be in $S$ because $S$ covers the edges $q_1z_{12}$ and $q_1z_{13}$. Now both $q_2$ and $q_3$ cannot simultaneously be in $S$, because size of $S$ is three. Therefore, the vertex $z_{23}$ must be in $S$ to cover the edge $q_2 z_{23}$ if $q_2\notin S$, and to cover the edge $q_3z_{23}$ if $q_3\notin S$. Therefore, the vertex $z_0$ cannot be in $S$ as $S$ has size $3$. Thus, both the endpoints of the edge $q_1z_0$ are not in $S$, which contradicts the fact that $S$ is a vertex cover of $H$. 

        \item \textbf{Case 2: $q_2\notin S$.}

        An argument symmetric to Case 1 shows that Case 2 cannot hold.

        \item \textbf{Case 3: $\{q_1,q_2\}\subset S$}

        Since $S$ is not equal to $Q$, $S$ does not contain $q_3$. Therefore, $S$ must contain both $z_{13}$ and $z_{23}$ to cover the edges $q_3z_{13}$ and $q_3z_{23}$. But this is not possible, as the size of $S$ is exactly three. 
    \end{itemize}

Thus, $Q$ is the only vertex cover of $H$ of size three. Similarly, it can be shown that $R$ is the only vertex cover of $G[R \cup U]$ of size three.
\qed
\end{proof}

We use the elementary structural facts above. Every minimum connected vertex cover contains $A\cup C$ and one private neighbor of each support vertex (\Cref{lem:triangle}); any vertex cover has at least three vertices in each of $Q\cup Z$ and $R\cup U$ (\Cref{lem:minthree}), and the only covers of these induced subgraphs of size three are $Q$ and $R$ (\Cref{lem:sizethree}).

We now use the facts in \Cref{lem:triangle}, \Cref{lem:minthree} and \Cref{lem:sizethree} to determine $cvc(G)$. 

\begin{lemma}\label{lem:cvc}
    The size of the smallest connected vertex cover of the graph $G$ is $19$.
\end{lemma}
\begin{proof}
    Using \Cref{lem:triangle}, it is clear that any minimum connected vertex cover of $G$ must contain all the $A$-vertices, all the $C$-vertices, and one private vertex from each of the six private triangles. Also, by \Cref{lem:minthree}, any minimum connected vertex cover of $G$ must contain at least three vertices from $Q\cup Z$ and at least three vertices from $R\cup U$. Therefore, any minimum-sized connected vertex cover of $G$ must contain at least $18$ vertices (six support vertices, six private vertices, three vertices from $Q \cup Z$ and three vertices from $R \cup U$).
    
    Now, suppose that there exists a minimum-sized connected vertex cover $S$ of $G$ which contains exactly $18$ vertices. Then by \Cref{lem:triangle}, $S$ must contain six support and six private vertices. Now $S$ can only contain six more vertices. Since $S$ must contain at least three vertices from $Q\cup Z$ and at least three vertices from $R \cup U$ by \Cref{lem:minthree}, $S$ contains exactly three vertices from $Q\cup Z$ and exactly three vertices from $R \cup U$. Now by \Cref{lem:sizethree}, the only vertex cover of $G[Q \cup Z]$ of size three is $Q$ and the only vertex cover of $G[R \cup U]$ of size three is $R$. Since $S$ must also cover all the edges in $G[Q\cup Z]$ and $G[R\cup U]$, $S$ must contain only the vertices in $Q\cup R$ and no other vertices from $Q\cup R \cup Z\cup U$. Thus, we have shown that $S$ must contain all the $A$-vertices, all the $C$-vertices, six private vertices, all the $Q$-vertices, all the $R$-vertices, and no other vertex. 

    Now we show that $G[S]$ is not connected. Consider the set $S^\prime \subset S$ defined as $S^\prime:= \{a_1,q_1,r_1,c_1\}$. It can be seen that there is no edge from any vertex in $S^\prime$ to any other vertex in $S$ except the private neighbors of $a_1$ and $c_1$ which are in $S$ (say $p_1$ and $p_2$). Therefore, $G[S^\prime \cup\{p_1,p_2\}]$ forms a non-trivial connected component of $G[S]$. Thus, $S$ is not a connected vertex cover of $G$.

    Therefore, the size of any connected vertex cover of $G$ is at least $19$. It can be checked that the union of $A\cup C \cup Z \cup R$ with one private neighbor of each support vertex; is a connected vertex cover of $G$ of size $19$. Thus, we have shown that $cvc(G)=19$.\qed
\end{proof}

\begin{lemma}\label{lem:sizefour}
    The only vertex covers of $G[Q\cup Z]$ of size four are $Q \cup \{z\}$ (where $z\in Z$), $Z$, and $Z^\prime:= \{q_1,q_2,z_{13},z_{23}\}$. Similarly, the only vertex covers of $G[R\cup U]$ of size four are $R \cup \{u\}$ (where $u\in U$), $U$, and $U^\prime:= \{r_1,r_2,u_{13},u_{23}\}$.
\end{lemma}

\begin{proof}
    As seen in the proof of \Cref{lem:minthree}, $H=G[Q\cup Z]$ is bipartite with bipartition $Q\uplus Z$. Since $Q$ is one side of the bipartition, it is clear that $Q\cup \{z\}$ is a vertex cover of $H$ of size four, for any $z\in Z$. Since $Z$ is one side of the bipartition, $Z$ is also a vertex cover of $H$ of size four. Let $S$ be a vertex cover of $G$ of size four which is not of the form $Q\cup \{z\}$ for any $z\in Z$, and also not equal to $Z$. We show that $S$ is exactly equal to $Z^\prime$.
    
    Suppose $q_3\in S$, then at least one of $q_1$ and $q_2$ must lie outside $S$. Suppose $q_1$ lies outside $S$, its neighbors $z_{12},z_{13}$ and $z_0$ must lie inside $S$. Since $S$ is of size four, $S=\{q_3,z_0,z_{12},z_{13}\}$. Therefore, the edge $q_2z_{23}$ is not covered by $S$, which contradicts the fact that $S$ is a vertex cover of $H$. A similar contradiction can be derived when $q_2$ does not belong to $S$. Hence, it can be seen that $q_3\in S$ is not possible and therefore, $q_3\notin S$. 

    Since $q_3\notin S$, the neighbors of $q_3$, i.e., $z_{13}$ and $z_{23}$ must belong to $S$. Since $S\neq Z$, at least one of $z_0$ and $z_{12}$ must be outside $S$. Now in order to cover the edge $q_1z_0$ (if $z_0$ lies outside $S$), or the edge $q_1z_{12}$ (if $z_{12}$ lies outside $S$), the vertex $q_1$ must belong to $S$. A similar argument can be used to show that the vertex $q_2$ must belong to $S$. Since the size of $S$ is four, $S$ is equal to $\{q_1,q_2,z_{13},z_{23}\}$, that is, $S$ must be equal to $Z^\prime$. \qed
\end{proof}

The size-four covers of the two induced subgraphs are classified in \Cref{lem:sizefour}. In particular, $Z^\prime=\{q_1,q_2,z_{13},z_{23}\}$ and $U^\prime=\{r_1,r_2,u_{13},u_{23}\}$.

\begin{lemma}\label{lem:four-types-G}
    There are exactly four types of connected vertex covers of $G$ of size $19$:
    \begin{enumerate}
        \item \textbf{Type A:} The sets formed by $A\cup C\cup R \cup Z$ and exactly one private neighbor of each support vertex. 
        \item \textbf{Type B:} The sets formed by $A\cup C\cup R \cup Z^\prime$ and exactly one private neighbor of each support vertex. 
        \item \textbf{Type C:} The sets formed by $A\cup C\cup Q\cup U$ and exactly one private neighbor of each support vertex. 
        \item \textbf{Type D:} The sets formed by $A\cup C\cup Q\cup U^\prime$ and exactly one private neighbor of each support vertex. 
    \end{enumerate}
Hence, every vertex of $G$ belongs to some minimum-sized connected vertex cover of $G$.
\end{lemma}
\begin{proof}[of \Cref{lem:four-types-G}]
It can be verified that all the subsets of $V$ which are of any of the four types described above, indeed form a vertex cover of $G$. 
    
    As seen in \Cref{lem:triangle}, every connected vertex cover of $G$ must contain the following twelve vertices: all the six support vertices ($A\cup C$), and one private vertex adjacent to each support vertex. Now consider a minimum-sized and connected vertex cover $S$ of $G$. Since $S$ has twelve vertices as described above, it has exactly seven more vertices. Also, by \Cref{lem:minthree}, $S$ must also have at least three vertices from $Q\cup Z$ and at least three vertices from $R\cup U$. 

    Now there are two possibilities: $S$ contains four vertices from $Q\cup Z$ and three vertices from $R\cup U$, or vice versa. First consider the case where $S$ contains four vertices from $Q\cup Z$ and three vertices from $R\cup U$. Since $S$ contains exactly three vertices from $R\cup U$, and $S$ must also cover all the edges in $G[R\cup U]$, by \Cref{lem:sizethree}, $S$ must contain exactly the three vertices in $R$ from $R\cup U$. 

    Since $S$ must contain exactly four vertices from $Q\cup Z$ and $S$ must also cover all the edges of $G[Q\cup Z]$, then $S\cap (Q\cup Z)$ must be either equal to $Q\cup \{z\}$ (for some $z\in Z$), or $Z$, or $Z^\prime$, by \Cref{lem:sizefour}. We show that it is not possible to have $S\cap (Q\cup Z)$ equal to $Q\cup \{z\}$, for any $z\in Z$.

    We show that if $S\cap (Q\cup Z)$ equal to $Q\cup \{z\}$, for some $z\in Z$, then $G[S]$ is not connected. Note that $S$ contains all the support vertices ($A$-vertices and $C$-vertices), one private vertex adjacent to each support vertex, all the vertices in $R$. Note that $S$ does not contain any vertex in $U$. Now consider the following cases:
    \begin{itemize}
        \item \textbf{Case 1:} $S\cap (Q\cup Z)$ equal to $Q\cup \{z_{13}\}$. 
        
        Consider the set $S^\prime$ consisting of the vertices $\{a_2,q_2,r_2,c_2\}$ and the private neighbors of $a_2$ and $c_2$ which are in $S$. The graph $G[S^\prime]$ forms a non-trivial connected component of $S$, as any vertex in $S\setminus S^\prime$ is not adjacent to any vertex in $S$. This is because the vertices outside $S^\prime$ which are adjacent to any vertex in $S^\prime$ are either private neighbors of $a_2$ and $c_2$ (which are not in $S$), or $U$-vertices, or vertices in $Z\setminus \{z_{13}\}$. All of these vertices are not in $S$.  Thus, $G[S]$ is not connected.
         \item \textbf{Case 2:} $S\cap (Q\cup Z)$ equal to $Q\cup \{z_{23}\}$.

         An argument analogous to Case 1 holds in this case for the set $S^\prime$ consisting of the vertices $\{a_1,q_1,r_1,c_1\}$ and the private neighbors of $a_1$ and $c_1$ which are in $S$.
         
         \item \textbf{Case 3:} Neither $z_{13}$ nor $z_{23}$ belongs to $S$.

          An argument analogous to Case 1 holds in this case for the set $S^\prime$ consisting of the vertices $\{a_3,q_3,r_3,c_3\}$ and the private neighbors of $a_3$ and $c_3$ which are in $S$.  
    \end{itemize}
Thus, we have shown that for any minimum-sized connected vertex cover $S$ of $G$, such that $S$ contains four vertices from $Q\cup Z$ and three vertices from $R\cup U$, $S\cap (R\cup U \cup Q\cup Z)= R\cup Z$, or $S\cap (R\cup U \cup Q\cup Z)= R\cup Z^\prime$. This means that $S$ is either of Type $A$ or Type $B$, respectively.

Similarly, it can be shown that for any minimum-sized connected vertex cover $S$ of $G$, such that $S$ contains four vertices from $R \cup U$ and three vertices from $Q\cup Z$, $S\cap (R\cup U \cup Q\cup Z)= Q\cup U$, or $S\cap (R\cup U \cup Q\cup Z)= Q\cup U^\prime$. This means that $S$ is either of Type $C$ or Type $D$, respectively. 

Also, it can be seen that any vertex cover of each of the four types given above is connected. Consider a vertex cover $S$ of Type $A$. It is clear that any two $A$-vertices, say $a_i$ and $a_j$ are adjacent to $z_{ij}$ which is also in $S$. Hence, there exists a path which lies completely in $S$ between any two $A$-vertices. Each vertex in $R$ is adjacent to a vertex in $A$. Each vertex in $C$ is adjacent to a vertex in $R$. Each private vertex is adjacent to a vertex in $A$ or a vertex in $C$. Each vertex in $Z$ is adjacent to a vertex in $A$. Hence, there exists a path between any two vertices of $S$, which lies completely in $S$. Thus, a vertex cover of Type $A$ is connected. A similar proof can be used to show that a vertex cover of Type $C$ is connected.

Now consider a vertex cover $T$ of Type $B$. Recall that $T\cap (R\cup U\cup Q\cup Z)= R\cup Z^\prime$, where $Z^\prime$ is equal to $\{q_1,q_2,z_{13},z_{23}\}$. The vertices $a_1$ and $a_3$ are adjacent to $z_{13}$ which belongs to $T$, and the vertices $a_2$ and $a_3$ are adjacent to $z_{23}$ which belongs to $T$. Therefore, there exists a path which lies completely inside $T$, between any two $A$-vertices. Now, similar to the case above, each vertex in $R$ is adjacent to a vertex in $A$, each vertex in $C$ is adjacent to a vertex in $R$, and each private vertex is adjacent to a vertex in $A\cup C$. The vertices $q_1$ and $q_2$ are adjacent to the vertices $a_1$ and $a_2$, respectively. The vertices $z_{13}$ and $z_{23}$ are adjacent to the vertices $q_1$ and $q_2$, respectively. That is, there exists a path between any two vertices of $T$, which lies completely inside $T$. Thus, a vertex cover of Type $B$ is connected. A similar proof can be used to show that a vertex cover of Type $D$ is connected.

It is clear that each private vertex belongs to a connected vertex cover. Each $A$-vertex and each $C$-vertex belong to every connected vertex cover. Each $R$-vertex belongs to a Type A, or a Type B vertex cover. Each $Q$-vertex belongs to a Type C, or a Type D vertex cover. Each $Z$-vertex belongs to a Type A vertex cover, and each $U$-vertex belongs to a Type C vertex cover. Hence, every vertex of $G$ belongs to some minimum-sized connected vertex cover of $G$. 
\qed
\end{proof}

We have shown that every vertex of $G$ belongs to some minimum-sized connected vertex cover of $G$. Now we show that $ecvc(G)$ is not equal to $cvc(G)$, in the next two lemmas.  

\begin{lemma}\label{lem:typeconfig}
Consider the \emph{Eternal Connected Vertex Cover} game on $G$ with $19$ guards. If the initial configuration of guards does not form a connected vertex cover of Type A, the attacker can force the defender to arrange the guards in a Type A configuration in a single move (or win the game).
\end{lemma}

\begin{proof}
    If the initial configuration of the guards is not a connected vertex cover of $G$, then the attacker wins trivially. Assume that the initial configuration of guards is a connected vertex cover of $G$, that is not Type A. The attacker attacks an edge adjacent to $z_0$. Since $z_0$ does not belong to any vertex cover of Type B, C or D, initially there is no guard on $z_0$. Hence, after the attack, the guard on the other end-point of the attacked edge is forced to move to $z_0$, and cannot move any further before the next attack. If the next configuration of the guards is not a connected vertex cover, the attacker wins. If the next configuration formed by the guards is a connected vertex cover of the graph $G$, it must be a connected vertex cover of Type A, as these are the only vertex covers of $G$ which contain the vertex $z_0$.
    \qed
\end{proof}

\begin{lemma}\label{lem:attack}
    If the guards form a Type A configuration on $G$, the attacker can win the game in the next move.
\end{lemma}

    \begin{proof}
Suppose that the guards form a Type A configuration and attack the edge $c_1u_0$.
Since $u_0$ is unoccupied, any legal response must place a guard
on $u_0$. By \Cref{lem:four-types-G}, the resulting connected
vertex cover of size $19$, if one exists, must be of Type C.

Let $P_C$ be the six private neighbors of the vertices in $C$,
and let $X=C\cup U\cup P_C$. Every neighbor of $X$ outside $X$
belongs to $Q\cup R$. Initially, a Type A configuration has six
guards in $X$ (three on $C$ and one private neighbor of each
vertex in $C$), three guards in $R$, and none in $Q$. Hence at
most nine guards can occupy $X$ after one round.

But a Type C configuration has ten guards in $X$: three on $C$,
four on $U$, and one private neighbor of each vertex in $C$.
It therefore cannot be reached in one round. The attack on
$c_1u_0$ has no legal response.
\qed\end{proof}

It can be seen from \Cref{lem:typeconfig} and \Cref{lem:attack}, that for any initial configuration on $G$ with $19$ guards, the attacker can force the guards to move to a Type A configuration (if they already do not form a Type A configuration), or the attacker wins. Once the guards form a Type A configuration, the attacker attacks the edge $c_1u_0$ and wins. Hence, in any configuration on $G$ with $19$ guards, the defender cannot win in the \emph{Eternal Connected Vertex Cover} game. Thus, the eternal connected vertex cover number of $G$ ($ecvc(G)$) is not equal to the size of the minimum connected vertex cover of $G$ ($cvc(G)$). From \Cref{lem:four-types-G}, it is clear that every vertex of $G$ belongs to some minimum-sized connected vertex cover of $G$. Hence, we have the main result of this section.

\begin{theorem}\label{counterexamplethm}
    There exists a graph $G$ such that every vertex of $G$ belongs to a minimum-sized connected vertex cover of $G$, but $ecvc(G)\neq cvc(G)$.
\end{theorem}

\section{Modification of the counterexample}\label{sec:modified-example}

In the last section, we saw that there is a graph $G$ such that every vertex of $G$ belongs to some minimum-sized connected vertex cover of $G$, but $ecvc(G)\neq cvc(G)$. This is because every configuration with $cvc(G)$ many guards can be defeated by the attacker in a finite sequence of attacks. But the natural question also arises that whenever $ecvc(G)=cvc(G)$, does this mean that every minimum-sized connected vertex cover is a winning configuration?
Here, we answer this question also in the negative.

Consider the graph $G^\prime$ such that $V(G^\prime)=V(G)$, and $E(G^\prime)=E(G)\setminus \{q_1z_{13}\}$. We show that $ecvc(G^\prime)=cvc(G^\prime)$, but not all the minimum-sized connected vertex covers of $G^\prime$ are winning configurations for the defender on $G^\prime$ in the \emph{Eternal Connected Vertex Cover} game. First, we show that $ecvc(G^\prime)=cvc(G^\prime)$, and for this we first understand the structure of the minimum-sized connected vertex covers of $G^\prime$.

\begin{lemma}\label{lem:cvcnew}
    The size of the smallest connected vertex cover of the graph $G^\prime$ is $19$.
\end{lemma}

\begin{proof}
    Let $S$ be a minimum-sized connected vertex cover of $G^\prime$. It is clear that \Cref{lem:triangle} holds for the graph $G^\prime$ as well, as the vertices in $A\cup C$ are cut-vertices of $G^\prime$ as well, and must be present in every connected vertex cover. Also, one private neighbor of each support vertex must be present, to cover the edge between the two private vertices, in each private triangle. Thus, $S$ must contain exactly twelve vertices from the union of $A,C$ and the set of private vertices. Also the graph $G^\prime[R\cup U]$ has the matching $\{r_1u_{12},r_2u_{23},r_3u_{13}\}$, and the graph $G^\prime[Q\cup Z]$ has the matching $\{q_1z_{12},q_2z_{23},q_3z_{13}\}$. Therefore, $S$ must contain  at least three vertices from $R\cup U$ and at least three vertices from $Q\cup Z$. From \Cref{lem:sizethree}, the only vertex cover of $G^\prime[R\cup U]$ of size three is $R$, as the edges of $R\cup U$ are preserved in $G^\prime$. 

    \begin{claim}
    The only vertex cover of $G^\prime[Q\cup Z]$ of size three is $Q$.
    \end{claim}
    \begin{proof}
    It is clear that the graph $H^\prime:=G^\prime[Q\cup Z]$ is bipartite with the bipartition $Q\uplus Z$. Let $S$ be a vertex cover of size three of $H^\prime$ such that $S\neq Q$. Now suppose that $q_1\notin S$. Since $S$ is a vertex cover of $H^\prime$, the neighbors of $q_1$, i.e., $z_0,z_{12}$ must be in $S$. Now $\{q_2z_{23},q_3z_{13}\}$ is matching of size two, hence $S$ must contain at least two more vertices from $\{q_2,z_{23},q_3,z_{13}\}$. But this is not possible, as the size of $S$ is three. Thus, $q_1$ must belong to $S$, and since $S\neq Q$, either $q_2$ or $q_3$ does not belong to $S$.

    Suppose $q_2\notin S$, the neighbors of $q_2$, i.e., $z_0,z_{12}$ must be in $S$. Since $S$ already contains $q_1$, $S$ must be equal to $\{q_1,z_0,z_{12}\}$, but then both the endpoints of the edge $q_2z_{23}$ are outside $S$, contradicting the fact that $S$ is a vertex cover of $H^\prime$.

    Therefore, $q_1,q_2$ belong to $S$, and $q_3$ does not belong to $S$. Since the size of $S$ is three, exactly one vertex of $Z$ must belong to $S$. Therefore, at least one of $z_{13}$ and $z_{23}$ does not belong to $S$, and hence, at least one of the edges $q_3z_{13}$ and $q_3z_{23}$ has both its endpoints outside $S$. This again contradicts the fact that $S$ is a vertex cover of $H^\prime$.

    Hence, the only vertex cover of $H^\prime$ of size three is $Q$. \qed
    \end{proof}
We have shown above that any minimum-sized connected vertex cover of $G^\prime$ must contain twelve vertices from the union of $A,C$ and the set of private vertices, at least three vertices from $R\cup U$, and at least three vertices from $Q\cup Z$. Consider a minimum-sized connected vertex cover $S$ of $G^\prime$ which has size exactly $18$. Therefore, $S$ must contain twelve vertices from the union of $A,C$ and the set of private vertices, all the vertices in $Q$, and all the vertices in $R$ (and no other vertices). But now it can be seen that $G'[S]$ is not connected. It was shown in the proof of \Cref{lem:cvc} that $G[S]$ is not connected, and $V(G^\prime)=V(G)$, $E(G^\prime)\subsetneq E(G)$, which implies that $G^\prime[S]$ is not connected. 

Thus, any minimum connected vertex cover of $G^\prime$ must have size at least $19$. It can be seen that the union of $A\cup C \cup Z \cup R$ with one private neighbor of each support vertex; is a connected vertex cover of $G^\prime$ of size $19$. Thus, we have shown that $cvc(G^\prime)=19$.\qed
\end{proof}

Next, we understand the structure of all the connected vertex covers of $G^\prime$ of size $19$.

\begin{lemma}\label{lem:sizefournew}
    The only vertex covers of $G^\prime[Q\cup Z]$ of size four are $Q \cup \{z\}$ (where $z\in Z$), $Z$, $Z^\prime:= \{q_1,q_2,z_{13},z_{23}\}$, $Z_1=\{z_0,z_{12},z_{23},q_3\}$ and $Z_2=\{z_0,z_{12},q_2,q_3\}$. Also, the only vertex covers of $G^\prime[R\cup U]$ of size four are $R \cup \{u\}$ (where $u\in U$), $U$, and $U^\prime:= \{r_1,r_2,u_{13},u_{23}\}$.
\end{lemma}
\begin{proof}
    Since $H^\prime=G^\prime[Q\cup Z]$ is a bipartite graph with bipartition $(Q\uplus Z)$, it is clear that $Q\cup\{z\}$ is a vertex cover of $H^\prime$ of size four. Now, let $S$ be a vertex cover of $H^\prime$ such that $Q\not\subset S$. Therefore, there exists $q\in Q$ such that $q\notin S$.
    
    Suppose $q_3\notin S$, then the neighbors of $q_3$, that is $z_{13}$ and $z_{23}$ must be in $S$. Recall that the size of $S$ is four, and hence $S$ can contain only two more vertices. The graph $H^\prime[\{q_1,q_2,z_0,z_{12}\}]$ is the graph $K_{2,2}$ and hence its only two vertex covers of size two are $\{q_1,q_2\}$ and $\{z_0,z_{12}\}$. It can indeed be verified that $Z^\prime=\{z_{23},z_{13},q_1,q_2\}$ is a vertex cover of $H^\prime$ of size four. Also, as $H^\prime$ is bipartite, with bipartition $Q\uplus Z$, $Z$ is clearly a vertex cover of $H^\prime$ of size four. Therefore, the only vertex covers of $H^\prime$ of size four which do not contain $q_3$ are $Z$ and $Z^\prime$. 

    Now consider the case where $q_3\in S$ and $q_2\notin S$. Since, $q_2\notin S$, its neighbors $z_0,z_{12},z_{23}$ must be in $S$. Since $S$ is of size four, $S=Z_1=\{z_0,z_{12},z_{23},q_3\}$. It can be verified that $Z_1$ is indeed a vertex cover of $H^\prime$.

    Finally, consider the case where $q_3\in S$ and $q_2\in S$. Since $Q\not\subset S$, the vertex $q_1$ does not belong to $S$. Therefore, the vertices $z_0$ and $z_{12}$ must belong to $S$ because $z_0$ and $z_{12}$ are the neighbors of $q_1$ which is not in $S$. Since the size of $S$ is four, $S=\{q_2,q_3,z_0,z_{12}\}=Z_2$. It can be verified that $Z_2$ is indeed a vertex cover of $H^\prime$. 
    
    The graph $G^\prime[R\cup U]$ is exactly the same as the graph $G[R\cup U]$, hence from the proof of \Cref{lem:sizefour}, it follows that the only vertex covers of $G^\prime[R\cup U]$ of size four are $R \cup \{u\}$ (where $u\in U$), $U$, and $U^\prime:= \{r_1,r_2,u_{13},u_{23}\}$.
    \qed
\end{proof}

\begin{lemma}\label{lem:typeconfignew}
    There are exactly five types of connected vertex covers of $G^\prime$ of size $19$:
    \begin{enumerate}
        \item \textbf{Type A:} The sets formed by $A\cup C\cup R \cup Z$ and exactly one private neighbor of each support vertex. 
        \item \textbf{Type B:} The sets formed by $A\cup C\cup R \cup Z^\prime$ and exactly one private neighbor of each support vertex. 
        \item \textbf{Type C:} The sets formed by $A\cup C\cup Q\cup U$ and exactly one private neighbor of each support vertex. 
        \item \textbf{Type D:} The sets formed by $A\cup C\cup Q\cup U^\prime$ and exactly one private neighbor of each support vertex. 
        \item \textbf{Type E:} The sets formed by $A\cup C\cup R \cup Z_1$ and exactly one private neighbor of each support vertex. 
    \end{enumerate}
Hence, every vertex of $G^\prime$ belongs to some minimum-sized connected vertex cover of $G^\prime$.
\end{lemma}

\begin{proof}
    As seen in the proof of \Cref{lem:cvcnew}, every minimum-sized connected vertex cover of $G^\prime$ must contain all the vertices in $A\cup C$, and one private vertex from each private triangle. Also, every minimum-sized connected vertex cover of $G^\prime$ must contain at least three vertices from $R\cup U$ and at least three vertices from $Q\cup Z$. Let $S$ be a minimum-sized connected vertex cover of $G^\prime$, the size of $S$ is $19$ as seen in \Cref{lem:cvcnew}. Therefore, $S$ must contain all the vertices in $A\cup C$, one private vertex from each private triangle, three vertices from $R\cup U$ and four vertices from $Q\cup Z$; or all the vertices in $A\cup C$, one private vertex from each private triangle, four vertices from $R\cup U$ and three vertices from $Q\cup Z$.

    Recall from the proof of \Cref{lem:sizefournew} that the only vertex covers of size four of $G^\prime[R\cup U]$ are $R\cup\{u\}$ (for any $u\in U$), $U$ and $U^\prime$; and only vertex covers of size four of $G^\prime[Q\cup Z]$ are $Q\cup\{z\}$ (for any $z\in Z$), $Z,Z^\prime$, $Z_1$ and $Z_2$.

    As seen in the proof of \Cref{lem:four-types-G}, $S\cap (R\cup U \cup Z\cup Q)$ cannot be equal to $R\cup (Q\cup \{z\})$, for any $z\in Z$. Also, $S\cap (R\cup U \cup Z\cup Q)$ cannot be equal to $Q\cup (R\cup \{u\})$, for any $u\in U$. This is because $G[S]$ is not connected in these cases. This means that $G^\prime[S]$ is also not connected, because $V(G^\prime)=V(G)$ and $E(G^\prime)=E(G)\setminus \{q_1z_{13}\}$. 

    If $S\cap(R \cup U \cup Z \cup Q)=R\cup Z_2=R\cup\{z_0,z_{12},q_2,q_3\}$, consider the set $K=\{a_3,c_3,q_3,r_3,p_1,p_2\}$, where $p_1,p_2$ are the private neighbors of $a_3$ and $c_3$ in $S$, respectively. The graph $G^\prime[K]$ is connected, and no edge of $G^\prime[S]$ joins $K$ to $S\setminus K$. Thus, $G^\prime[S]$ is disconnected.

    Therefore, any minimum-sized connected vertex cover of $G^\prime$ is one of the following types:

    \begin{enumerate}
        \item \textbf{Type A:} The sets formed by $A\cup C\cup R \cup Z$ and exactly one private neighbor of each support vertex. 
        \item \textbf{Type B:} The sets formed by $A\cup C\cup R \cup Z^\prime$ and exactly one private neighbor of each support vertex. 
        \item \textbf{Type C:} The sets formed by $A\cup C\cup Q\cup U$ and exactly one private neighbor of each support vertex. 
        \item \textbf{Type D:} The sets formed by $A\cup C\cup Q\cup U^\prime$ and exactly one private neighbor of each support vertex. 
        \item \textbf{Type E:} The sets formed by $A\cup C\cup R \cup Z_1$ and exactly one private neighbor of each support vertex. 
    \end{enumerate}

It can be verified that each of these sets are connected, and indeed form vertex covers of $G^\prime$. \qed
\end{proof}

Now we are ready to show that $cvc(G^\prime)$ guards are sufficient for the defender to defend an infinite sequence of attacks on the graph $G^\prime$.

\begin{lemma}\label{lem:lemma13}
    For the graph $G^\prime$, we have $ecvc(G^\prime)=cvc(G^\prime)$ and any sequence of attacks can be defended by using only configurations of Type B, Type C, Type D and Type E. 
\end{lemma}

\begin{proof}
There exists a winning strategy for the defender, where the defender starts with one of the configurations of Type B, Type C, Type D, and Type E; and end up at one of the configurations of Type B, Type C, Type D, and Type E after each attack.

If both endpoints of the attacked edge are occupied, interchange their guards. This traverses the attacked edge and leaves the occupied set unchanged.

\textit{Attacks involving a private vertex.}

Let $s$ be a support vertex, and let $p,p'$ be its private
neighbours, with $p$ occupied and $p'$ unoccupied.

For an attack on $pp'$, move the guard from $p$ to $p'$.
For an attack on $sp'$, simultaneously move
\[
s\longrightarrow p',
\qquad
p\longrightarrow s.
\]

All the other guards stay. In both cases, the only change in the
occupied set is the replacement of $p$ by $p'$. Thus, the
type is unchanged.

Now we describe a strategy of the defender when the guards occupy a Type B configuration. The attacker attacks an edge with one endpoint unoccupied, and no endpoint being a private vertex. Recall that a Type B configuration consists of all the vertices in $A\cup C$, and a private neighbor of each support vertex, and all the vertices in $R\cup Z^\prime$ where $Z^\prime=\{z_{23},z_{13},q_1,q_2\}$. A Type C configuration consists of all the vertices in $A\cup C \cup Q\cup U$, and a private neighbor of each support vertex. A Type E configuration consists of all the vertices in $A\cup C$, and a private neighbor of each support vertex, and all the vertices in $R\cup Z_1$ where $Z_1=\{z_0,z_{12},z_{23},q_3\}$. In the strategy described below, no guard on a private vertex moves. Also, if a vertex with a guard is not mentioned, the guard on this vertex stays on his position. The notation $v_1\rightarrow v_2\rightarrow\ldots\rightarrow v_k$, where $v_1,v_2,\ldots v_{k-1}$ are occupied vertices and $v_k$ is an unoccupied vertex means that the guard on $v_{k-1}$ moves to $v_k$, the guard on $v_{k-2}$ moves to $v_{k-1}$, and so on such that the guard on $v_1$ moves to $v_2$. At the end of this movement, the vertices $v_2,v_3,\ldots,v_k$ are occupied and the vertex $v_1$ is unoccupied.

\begin{enumerate}
    \item The unoccupied endpoint of the attacked edge is $q_3,z_0$ or $z_{12}$. In these cases, the guards rearrange themselves from a Type B to a Type E configuration.
    \begin{enumerate}
        \item The unoccupied endpoint of the attacked edge is $q_3$.
        \begin{enumerate}
            \item The attacked edge is $z_{13}q_3$:
            
             $z_{13}\rightarrow q_3$.\\
            $q_1\rightarrow z_0$.\\
            $q_2\rightarrow z_{12}$.
           
            \item The attacked edge is $z_{23}q_3$:

            $q_2\rightarrow z_{23}\rightarrow q_3$.\\
            $q_1\rightarrow z_{12}$.\\
            $z_{13}\rightarrow a_1 \rightarrow z_0$.
            
            \item The attacked edge is $a_3q_3$:

            $z_{13}\rightarrow a_3 \rightarrow q_3$.\\
            $q_2 \rightarrow z_{12}$.\\
            $q_1\rightarrow z_0$.
            
            \item The attacked edge is $c_3q_3$:

            $z_{13}\rightarrow a_3 \rightarrow r_3 \rightarrow c_3 \rightarrow q_3$.\\
            $q_2\rightarrow z_0$.\\
            $q_1 \rightarrow z_{12}$.
        \end{enumerate}
        \item The unoccupied endpoint of the attacked edge is $z_0$.
        \begin{enumerate}
            \item The attacked edge is $q_1z_0$:

            Same as $1(a)(i)$ above.
        \item The attacked edge is $q_2z_0$:

             Same as $1(a)(iv)$ above.
            \item The attacked edge is $a_1z_0$:

            Same as $1(a)(ii)$ above.
        \end{enumerate}
        \item The unoccupied endpoint of the attacked edge is $z_{12}$.

        \begin{enumerate}
            \item The attacked edge is $q_1z_{12}$:

            Same as $1(a)(ii)$ above.
            
            \item The attacked edge is $q_2z_{12}$:

            Same as $1(a)(iii)$ above.
            
            \item The attacked edge is $a_1z_{12}$:

            $z_{13}\rightarrow a_1 \rightarrow z_{12}$.\\
            $q_1\rightarrow z_0$.\\
            $q_2\rightarrow z_{23} \rightarrow q_3$.
            
            \item The attacked edge is $a_2z_{12}$:

            $q_2 \rightarrow a_2 \rightarrow z_{12}$.\\
            $q_1 \rightarrow z_0$.\\
            $z_{13}\rightarrow q_3$.
        \end{enumerate}
    \end{enumerate}
    
    \item The unoccupied endpoint of the attacked edge is $u_0,u_{12},u_{23}$ or $u_{13}$. In these cases, the guards rearrange themselves from a Type B to a Type C configuration.
    \begin{enumerate}
        \item The unoccupied endpoint of the attacked edge is $u_0$.
        \begin{enumerate}
            \item The attacked edge is $c_1u_0$:

            $z_{13}\rightarrow a_1 \rightarrow q_1 \rightarrow c_1 \rightarrow u_0.$\\
            $r_1\rightarrow u_{13}$.\\
            $z_{23}\rightarrow q_3$.\\
            $r_2\rightarrow u_{12}$.\\
            $r_3\rightarrow u_{23}$.
            
            \item The attacked edge is $r_1u_0$:

            $r_1\rightarrow u_0$.\\
            $r_2\rightarrow u_{12}$.\\
            $r_3\rightarrow u_{13}$.\\
            $z_{23}\rightarrow q_2 \rightarrow c_2 \rightarrow u_{23}$.\\
            $z_{13}\rightarrow q_3$.
            
            \item The attacked edge is $r_2u_0$:

            $r_2\rightarrow u_0$.\\
            $r_1\rightarrow u_{12}$.\\
            $r_3\rightarrow u_{13}$.\\
            $z_{23}\rightarrow q_2 \rightarrow c_2 \rightarrow u_{23}$.\\
            $z_{13}\rightarrow q_3$.

        \end{enumerate}
        \item The unoccupied endpoint of the attacked edge is $u_{12}$.
        \begin{enumerate}
            \item The attacked edge is $r_1u_{12}$:

            Same as $2(a)(iii)$ above.
            
            \item The attacked edge is $r_2u_{12}$:

             Same as $2(a)(ii)$ above.
            
            \item The attacked edge is $c_1u_{12}$:

            $r_1\rightarrow c_1 \rightarrow u_{12}$.\\
            $r_2\rightarrow u_0$.\\
            $r_3\rightarrow u_{13}$.\\
            $z_{23}\rightarrow q_2 \rightarrow c_2 \rightarrow u_{23}$.\\
            $z_{13}\rightarrow q_3$.
            
            \item The attacked edge is $c_2u_{12}$:
            
            $z_{23}\rightarrow q_2 \rightarrow c_2\rightarrow u_{12}$.\\
            $r_2 \rightarrow u_{23}$.\\
            $r_1\rightarrow u_0$.\\
            $z_{13}\rightarrow q_3$.\\
            $r_3\rightarrow u_{13}$.
        \end{enumerate}
        \item The unoccupied endpoint of the attacked edge is $u_{23}$.
        \begin{enumerate}
            \item The attacked edge is $r_2u_{23}$:

            $r_2 \rightarrow u_{23}$.\\
            $r_1\rightarrow u_0$.\\
            $z_{23}\rightarrow q_2 \rightarrow c_2\rightarrow u_{12}$.\\
            $z_{13}\rightarrow q_3$.\\
            $r_3\rightarrow u_{13}$.
            
            \item The attacked edge is $r_3u_{23}$:

            Same as $2 (a)(i)$ above.
            
            \item The attacked edge is $c_2u_{23}$:

             Same as $2 (a)(iii)$ above.
             
            \item The attacked edge is $c_3u_{23}$:

            $r_3\rightarrow c_3 \rightarrow u_{23}$.\\
            $z_{13}\rightarrow q_3$.\\
            $r_2\rightarrow u_0$.\\
            $z_{23}\rightarrow q_2 \rightarrow c_2 \rightarrow u_{12}$.\\
            $r_1\rightarrow u_{13}$.
        \end{enumerate}
        \item The unoccupied endpoint of the attacked edge is $u_{13}$.
        \begin{enumerate}
            \item The attacked edge is $r_1u_{13}$:

            Same as $2 (c) (iv)$ above.
            \item The attacked edge is $r_3u_{13}$:

            Same as $2 (c) (i)$ above.
            \item The attacked edge is $c_1u_{13}$:

            $r_1\rightarrow c_1 \rightarrow u_{13}$.\\
            $r_2\rightarrow u_0$.\\
            $r_3\rightarrow u_{23}$.\\
            $z_{23}\rightarrow q_2 \rightarrow c_2 \rightarrow u_{12}$.\\
            $z_{13}\rightarrow q_3$.
            
            \item The attacked edge is $c_3u_{13}$:

            $r_3\rightarrow c_3 \rightarrow u_{13}$.\\
            $z_{23}\rightarrow q_3$.\\
            $z_{13}\rightarrow a_1\rightarrow q_1\rightarrow c_1\rightarrow u_0$.\\
            $r_1\rightarrow u_{12}$.\\
            $r_2\rightarrow u_{23}$.
            
        \end{enumerate}
    \end{enumerate}
\end{enumerate}

Now we describe a strategy of the defender when the guards occupy a configuration of Type C, Type D or Type E. As before, the attacker attacks an edge with one endpoint unoccupied, and no endpoint being a private vertex. We cannot argue by symmetry here, since deleting the edge $q_1z_{13}$ destroys the automorphism of $G$ that exchanges $a_i$ with $c_i$, $q_i$ with $r_i$, and $Z$ with $U$. So we list the responses explicitly, in one table for each type. Each row of a table describes one simultaneous movement of the guards, in the notation introduced above. The first column lists the attacked edges for which the row is used: for each listed edge $xy$, the vertex $x$ is occupied, the vertex $y$ is unoccupied, and the row contains the move of the guard on $x$ to $y$, so the attacked edge is traversed. The second column gives the type of the configuration at the end of the movement. As in the case of Type B, no guard on a private vertex moves, and if a vertex with a guard is not mentioned, the guard on this vertex stays on his position. Every edge with exactly one occupied endpoint and no private endpoint appears in exactly one row of the table of the current type; there are $28$, $25$ and $27$ such edges for Type C, Type D and Type E, respectively.

\vspace{2mm}

\noindent\textit{Type C configurations.}

\nopagebreak
Recall that a Type C configuration consists of all the vertices in $A\cup C\cup Q\cup U$, and a private neighbor of each support vertex.

\begin{center}
\small
\renewcommand{\arraystretch}{1.2}
\begin{tabular}{@{}>{\raggedright\arraybackslash}p{2.8cm}@{\hspace{4pt}}c@{\hspace{6pt}}>{\raggedright\arraybackslash}p{7.2cm}@{}}
\hline
Attacked edges & Target type & Moves (all other guards stay)\\
\hline
$a_2z_{23}$, $a_3z_{13}$, $c_1r_1$, $u_0r_2$, $u_{23}r_3$ & B & \mbox{$u_0\rightarrow r_2$}, \mbox{$u_{23}\rightarrow r_3$}, \mbox{$q_3\rightarrow a_3\rightarrow z_{13}$}, \mbox{$u_{13}\rightarrow c_1\rightarrow r_1$}, \mbox{$u_{12}\rightarrow c_2\rightarrow q_2\rightarrow a_2\rightarrow z_{23}$}\\
\hline
$a_3z_{23}$, $a_1z_{13}$, $u_{12}r_1$, $c_2r_2$, $u_{13}r_3$ & B & \mbox{$u_{12}\rightarrow r_1$}, \mbox{$u_{13}\rightarrow r_3$}, \mbox{$q_3\rightarrow a_3\rightarrow z_{23}$}, \mbox{$u_{23}\rightarrow c_2\rightarrow r_2$}, \mbox{$u_0\rightarrow c_1\rightarrow q_1\rightarrow a_1\rightarrow z_{13}$}\\
\hline
$q_2z_{23}$, $u_{13}r_1$, $u_{12}r_2$, $a_3r_3$ & B & \mbox{$u_{12}\rightarrow r_2$}, \mbox{$u_{13}\rightarrow r_1$}, \mbox{$q_3\rightarrow a_3\rightarrow r_3$}, \mbox{$u_0\rightarrow c_1\rightarrow q_1\rightarrow a_1\rightarrow z_{13}$}, \mbox{$u_{23}\rightarrow c_2\rightarrow q_2\rightarrow z_{23}$}\\
\hline
$q_3z_{13}$, $a_1r_1$, $u_{23}r_2$, $c_3r_3$ & B & \mbox{$q_3\rightarrow z_{13}$}, \mbox{$u_{23}\rightarrow r_2$}, \mbox{$u_{13}\rightarrow c_3\rightarrow r_3$}, \mbox{$u_0\rightarrow c_1\rightarrow q_1\rightarrow a_1\rightarrow r_1$}, \mbox{$u_{12}\rightarrow c_2\rightarrow q_2\rightarrow z_{23}$}\\
\hline
$q_3z_{23}$, $u_0r_1$, $a_2r_2$ & B & \mbox{$q_3\rightarrow z_{23}$}, \mbox{$u_0\rightarrow r_1$}, \mbox{$u_{13}\rightarrow r_3$}, \mbox{$u_{12}\rightarrow c_1\rightarrow q_1\rightarrow a_1\rightarrow z_{13}$}, \mbox{$u_{23}\rightarrow c_2\rightarrow q_2\rightarrow a_2\rightarrow r_2$}\\
\hline
$q_1z_0$, $q_2z_{12}$ & E & \mbox{$q_1\rightarrow z_0$}, \mbox{$q_2\rightarrow z_{12}$}, \mbox{$u_0\rightarrow r_2$}, \mbox{$u_{12}\rightarrow r_1$}, \mbox{$u_{23}\rightarrow r_3$}, \mbox{$u_{13}\rightarrow c_3\rightarrow q_3\rightarrow z_{23}$}\\
\hline
$q_2z_0$, $q_1z_{12}$ & E & \mbox{$q_1\rightarrow z_{12}$}, \mbox{$q_2\rightarrow z_0$}, \mbox{$u_0\rightarrow r_2$}, \mbox{$u_{12}\rightarrow r_1$}, \mbox{$u_{23}\rightarrow r_3$}, \mbox{$u_{13}\rightarrow c_3\rightarrow q_3\rightarrow z_{23}$}\\
\hline
$a_1z_{12}$ & E & \mbox{$q_2\rightarrow z_0$}, \mbox{$u_0\rightarrow r_2$}, \mbox{$u_{12}\rightarrow r_1$}, \mbox{$u_{23}\rightarrow r_3$}, \mbox{$q_1\rightarrow a_1\rightarrow z_{12}$}, \mbox{$u_{13}\rightarrow c_3\rightarrow q_3\rightarrow z_{23}$}\\
\hline
$a_1z_0$, $a_2z_{12}$ & E & \mbox{$u_0\rightarrow r_2$}, \mbox{$u_{12}\rightarrow r_1$}, \mbox{$u_{23}\rightarrow r_3$}, \mbox{$q_1\rightarrow a_1\rightarrow z_0$}, \mbox{$q_2\rightarrow a_2\rightarrow z_{12}$}, \mbox{$u_{13}\rightarrow c_3\rightarrow q_3\rightarrow z_{23}$}\\
\hline
\end{tabular}
\end{center}

\vspace{2mm}

\noindent\textit{Type D configurations.}

\nopagebreak
A Type D configuration consists of all the vertices in $A\cup C$, and a private neighbor of each support vertex, and all the vertices in $Q\cup U^\prime$ where $U^\prime=\{r_1,r_2,u_{13},u_{23}\}$.

\begin{center}
\small
\renewcommand{\arraystretch}{1.2}
\begin{tabular}{@{}>{\raggedright\arraybackslash}p{2.8cm}@{\hspace{4pt}}c@{\hspace{6pt}}>{\raggedright\arraybackslash}p{7.2cm}@{}}
\hline
Attacked edges & Target type & Moves (all other guards stay)\\
\hline
$q_2z_{23}$, $q_3z_{13}$ & B & \mbox{$q_3\rightarrow z_{13}$}, \mbox{$u_{13}\rightarrow r_3$}, \mbox{$u_{23}\rightarrow c_2\rightarrow q_2\rightarrow z_{23}$}\\
\hline
$q_3z_{23}$, $a_1z_{13}$ & B & \mbox{$q_3\rightarrow z_{23}$}, \mbox{$u_{23}\rightarrow r_3$}, \mbox{$u_{13}\rightarrow r_1\rightarrow a_1\rightarrow z_{13}$}\\
\hline
$a_2z_{23}$, $a_3z_{13}$ & B & \mbox{$u_{13}\rightarrow r_3$}, \mbox{$q_3\rightarrow a_3\rightarrow z_{13}$}, \mbox{$u_{23}\rightarrow r_2\rightarrow a_2\rightarrow z_{23}$}\\
\hline
$a_3z_{23}$ & B & \mbox{$u_{23}\rightarrow r_3$}, \mbox{$q_3\rightarrow a_3\rightarrow z_{23}$}, \mbox{$u_{13}\rightarrow r_1\rightarrow a_1\rightarrow z_{13}$}\\
\hline
$r_1u_0$, $r_2u_{12}$ & C & \mbox{$r_1\rightarrow u_0$}, \mbox{$r_2\rightarrow u_{12}$}\\
\hline
$r_2u_0$, $r_1u_{12}$ & C & \mbox{$r_1\rightarrow u_{12}$}, \mbox{$r_2\rightarrow u_0$}\\
\hline
$c_1u_{12}$ & C & \mbox{$r_2\rightarrow u_0$}, \mbox{$r_1\rightarrow c_1\rightarrow u_{12}$}\\
\hline
$c_1u_0$, $c_2u_{12}$ & C & \mbox{$r_1\rightarrow c_1\rightarrow u_0$}, \mbox{$r_2\rightarrow c_2\rightarrow u_{12}$}\\
\hline
$q_1z_0$, $a_1z_{12}$, $u_{23}r_3$ & E & \mbox{$q_1\rightarrow z_0$}, \mbox{$q_2\rightarrow z_{23}$}, \mbox{$u_{23}\rightarrow r_3$}, \mbox{$u_{13}\rightarrow r_1\rightarrow a_1\rightarrow z_{12}$}\\
\hline
$q_2z_0$, $q_1z_{12}$, $u_{13}r_3$ & E & \mbox{$q_1\rightarrow z_{12}$}, \mbox{$q_2\rightarrow z_0$}, \mbox{$u_{13}\rightarrow r_3$}, \mbox{$u_{23}\rightarrow c_3\rightarrow q_3\rightarrow z_{23}$}\\
\hline
$q_2z_{12}$, $c_3r_3$ & E & \mbox{$q_1\rightarrow z_0$}, \mbox{$q_2\rightarrow z_{12}$}, \mbox{$u_{13}\rightarrow c_3\rightarrow r_3$}, \mbox{$u_{23}\rightarrow r_2\rightarrow a_2\rightarrow z_{23}$}\\
\hline
$a_1z_0$, $a_2z_{12}$, $a_3r_3$ & E & \mbox{$q_2\rightarrow z_{23}$}, \mbox{$q_1\rightarrow a_1\rightarrow z_0$}, \mbox{$u_{23}\rightarrow r_2\rightarrow a_2\rightarrow z_{12}$}, \mbox{$u_{13}\rightarrow c_3\rightarrow q_3\rightarrow a_3\rightarrow r_3$}\\
\hline
\end{tabular}
\end{center}

\vspace{2mm}

\noindent\textit{Type E configurations.}

\nopagebreak
Recall that a Type E configuration consists of all the vertices in $A\cup C$, and a private neighbor of each support vertex, and all the vertices in $R\cup Z_1$ where $Z_1=\{z_0,z_{12},z_{23},q_3\}$.

\begin{center}
\small
\renewcommand{\arraystretch}{1.2}
\begin{tabular}{@{}>{\raggedright\arraybackslash}p{2.8cm}@{\hspace{4pt}}c@{\hspace{6pt}}>{\raggedright\arraybackslash}p{7.2cm}@{}}
\hline
Attacked edges & Target type & Moves (all other guards stay)\\
\hline
$z_0q_1$, $z_{12}q_2$, $q_3z_{13}$ & B & \mbox{$q_3\rightarrow z_{13}$}, \mbox{$z_0\rightarrow q_1$}, \mbox{$z_{12}\rightarrow q_2$}\\
\hline
$z_{12}q_1$, $z_0q_2$ & B & \mbox{$q_3\rightarrow z_{13}$}, \mbox{$z_0\rightarrow q_2$}, \mbox{$z_{12}\rightarrow q_1$}\\
\hline
$z_{23}q_2$, $a_1z_{13}$ & B & \mbox{$z_0\rightarrow q_1$}, \mbox{$q_3\rightarrow z_{23}\rightarrow q_2$}, \mbox{$z_{12}\rightarrow a_1\rightarrow z_{13}$}\\
\hline
$c_1q_1$, $a_2q_2$ & B & \mbox{$q_3\rightarrow z_{13}$}, \mbox{$z_{12}\rightarrow a_2\rightarrow q_2$}, \mbox{$z_0\rightarrow a_1\rightarrow r_1\rightarrow c_1\rightarrow q_1$}\\
\hline
$a_1q_1$, $c_2q_2$, $a_3z_{13}$ & B & \mbox{$q_3\rightarrow a_3\rightarrow z_{13}$}, \mbox{$z_0\rightarrow a_1\rightarrow q_1$}, \mbox{$z_{12}\rightarrow a_2\rightarrow r_2\rightarrow c_2\rightarrow q_2$}\\
\hline
$r_1u_0$, $r_2u_{12}$ & C & \mbox{$r_1\rightarrow u_0$}, \mbox{$r_2\rightarrow u_{12}$}, \mbox{$r_3\rightarrow u_{23}$}, \mbox{$z_0\rightarrow q_1$}, \mbox{$z_{12}\rightarrow q_2$}, \mbox{$z_{23}\rightarrow q_3\rightarrow c_3\rightarrow u_{13}$}\\
\hline
$r_2u_0$, $r_1u_{12}$ & C & \mbox{$r_1\rightarrow u_{12}$}, \mbox{$r_2\rightarrow u_0$}, \mbox{$r_3\rightarrow u_{23}$}, \mbox{$z_0\rightarrow q_1$}, \mbox{$z_{12}\rightarrow q_2$}, \mbox{$z_{23}\rightarrow q_3\rightarrow c_3\rightarrow u_{13}$}\\
\hline
$c_1u_{12}$ & C & \mbox{$r_2\rightarrow u_0$}, \mbox{$r_3\rightarrow u_{23}$}, \mbox{$z_0\rightarrow q_1$}, \mbox{$z_{12}\rightarrow q_2$}, \mbox{$r_1\rightarrow c_1\rightarrow u_{12}$}, \mbox{$z_{23}\rightarrow q_3\rightarrow c_3\rightarrow u_{13}$}\\
\hline
$c_1u_0$, $c_2u_{12}$ & C & \mbox{$r_3\rightarrow u_{23}$}, \mbox{$z_0\rightarrow q_1$}, \mbox{$z_{12}\rightarrow q_2$}, \mbox{$r_1\rightarrow c_1\rightarrow u_0$}, \mbox{$r_2\rightarrow c_2\rightarrow u_{12}$}, \mbox{$z_{23}\rightarrow q_3\rightarrow c_3\rightarrow u_{13}$}\\
\hline
$c_2u_{23}$, $r_3u_{13}$ & D & \mbox{$r_3\rightarrow u_{13}$}, \mbox{$z_0\rightarrow q_1$}, \mbox{$z_{12}\rightarrow q_2$}, \mbox{$z_{23}\rightarrow a_2\rightarrow r_2\rightarrow c_2\rightarrow u_{23}$}\\
\hline
$c_3u_{23}$, $r_1u_{13}$ & D & \mbox{$z_0\rightarrow q_1$}, \mbox{$z_{23}\rightarrow q_2$}, \mbox{$r_3\rightarrow c_3\rightarrow u_{23}$}, \mbox{$z_{12}\rightarrow a_1\rightarrow r_1\rightarrow u_{13}$}\\
\hline
$r_2u_{23}$, $c_3u_{13}$ & D & \mbox{$z_0\rightarrow q_1$}, \mbox{$z_{12}\rightarrow q_2$}, \mbox{$r_3\rightarrow c_3\rightarrow u_{13}$}, \mbox{$z_{23}\rightarrow a_2\rightarrow r_2\rightarrow u_{23}$}\\
\hline
$r_3u_{23}$, $c_1u_{13}$ & D & \mbox{$r_3\rightarrow u_{23}$}, \mbox{$z_0\rightarrow q_1$}, \mbox{$z_{23}\rightarrow q_2$}, \mbox{$z_{12}\rightarrow a_1\rightarrow r_1\rightarrow c_1\rightarrow u_{13}$}\\
\hline
\end{tabular}
\end{center}

Together with the rule for attacks on an edge with both endpoints occupied, and the rules for attacks involving a private vertex given above, this gives a legal response to every attack from every configuration of Type B, Type C, Type D and Type E. After each response, the guards again occupy a configuration of one of these types.

Thus, we have shown that $ecvc(G^\prime)=cvc(G^\prime)$.

\qed\end{proof}
\begin{lemma}\label{lem:type-A-loses-Gprime}
No Type A configuration of $G^\prime$ is eternally winning.
In fact, an attack on $c_1u_0$ has no legal response.
\end{lemma}
\begin{proof}
Let $P_C$ be the six private neighbors of the vertices in $C$, and set
$X=C\cup U\cup P_C$. Every neighbor of $X$ outside $X$ belongs to
$Q\cup R$. A Type A configuration has six guards in $X$ and three
in $R$, while $Q$ is unoccupied. Thus, at most nine guards can occupy
$X$ after one round.

Attack the edge $c_1u_0$. The vertex $u_0$ is initially unoccupied,
so any legal response must place a guard there. By
\Cref{lem:typeconfignew}, the resulting minimum connected vertex
cover would have to be Type C. But every Type C configuration has
ten guards in $X$: three on $C$, four on $U$, and one private
neighbor of each vertex in $C$. This is impossible after one round.
Hence, the attack has no legal response.
\qed\end{proof}

In particular, these results give $ecvc(G^\prime)=cvc(G^\prime)=19$.

\section{A structural criterion for equality}
\label{sec:structural-criterion}
\label{sec:sufficient-condition}

The counterexample in \Cref{sec:Counter-example} separates two properties that
might initially appear equivalent. Requiring every vertex to occur in some
minimum connected vertex cover guarantees the pointwise availability of
minimum configurations, but it does not ensure that these configurations can
be reached from one another by legal responses to prescribed attacks. We now
identify a structural setting in which this dynamic compatibility follows
from the geometry of the minimum covers.

\subsection{The tight counting class}

The starting point is a sharp lower bound obtained by counting the edges
inside a connected vertex cover and those incident with its complement. The
equality case naturally produces the structure needed later: the cover
induces a tree, while every omitted vertex has degree two.

\begin{lemma}\label{counting-bound}
Let $G$ be a connected graph with $n$ vertices, $m$ edges, and
$\delta(G)\geq 2$. Then
\begin{equation}\label{cvc-lower-bound}
    cvc(G)\geq 2n-m-1.
\end{equation}
Moreover, equality holds if and only if $G$ has a connected vertex cover $C$
such that $G[C]$ is a tree and every vertex in $V(G)\setminus C$ has degree
two. When equality holds, every minimum connected vertex cover has these two
properties.
\end{lemma}
\begin{proof}
Let $C$ be any connected vertex cover and put $I=V(G)\setminus C$.
Since $I$ is independent and $G[C]$ is connected,
\[
    m=|E(G[C])|+\sum_{v\in I}d_G(v)
      \geq (|C|-1)+2|I|
      =2n-|C|-1.
\]
This proves \eqref{cvc-lower-bound}. Equality in this calculation holds
exactly when $|E(G[C])|=|C|-1$ and $d_G(v)=2$ for every $v\in I$.
As $G[C]$ is connected, the first condition says precisely that $G[C]$ is a
tree. Thus any connected vertex cover with the stated properties attains the
lower bound and is minimum. Conversely, if the lower bound is attained,
applying the same calculation to any minimum connected vertex cover gives
both properties. \qed
\end{proof}

The equality case of \Cref{counting-bound} is precisely the setting in which
the connected part of a minimum cover has no redundant edge and every vertex
outside it contributes the smallest degree permitted by the hypothesis.
Accordingly, let $\mathcal F$ be the class of connected graphs $G$ with
$\delta(G)\geq 2$ for which equality holds in
\eqref{cvc-lower-bound}. By \Cref{counting-bound}, this is equivalent to the
existence of a connected vertex cover $C$ such that $G[C]$ is a tree and
every vertex of $V(G)\setminus C$ has degree two.

\subsection{Characterizing eternal equality}

Membership in $\mathcal F$ alone does not imply $ecvc(G)=cvc(G)$. Consider
the diamond $K_4-e$. Its two vertices of degree three form the unique minimum
connected vertex cover; they induce a tree, and the two omitted vertices have
degree two. Nevertheless, an attack from an occupied vertex to an omitted
vertex cannot be defended with two guards, because no alternative minimum
connected vertex cover is available.

Thus two ingredients must be distinguished: the tree-like structure provided
by membership in $\mathcal F$, and the availability of a minimum cover
containing each vertex. The next theorem shows that, within $\mathcal F$,
these ingredients fit together exactly. It also gives an equivalent cycle
condition.

\begin{theorem}\label{thm:equality-F}
Let $G\in\mathcal F$. The following conditions are equivalent:
\begin{enumerate}
    \item\label{eternal-equality} $ecvc(G)=cvc(G)$;
    \item\label{vertex-membership} every vertex of $G$ belongs to some
    minimum connected vertex cover of $G$;
    \item\label{cycle} every cycle of $G$ contains at least
    two vertices whose degree in $G$ is two.
\end{enumerate}
Whenever these conditions hold, every minimum connected vertex cover of $G$
is an eternally winning configuration.
\end{theorem}
\begin{proof}
\emph{\ref{eternal-equality}$\Rightarrow$\ref{vertex-membership}.}
First suppose that condition~\ref{eternal-equality} holds, and fix a winning
strategy using $cvc(G)$ guards. Every configuration reached by the strategy
is a minimum connected vertex cover. Let $v$ be any vertex. If $v$ is not
occupied in the initial configuration, attack an edge incident with $v$.
The other endpoint is occupied, and defending this attack places a guard on
$v$. Thus some minimum connected vertex cover contains $v$, proving
condition~\ref{vertex-membership}.

\emph{\ref{vertex-membership}$\Rightarrow$\ref{cycle}.}
Next suppose that condition~\ref{vertex-membership} holds. By
\Cref{counting-bound}, every minimum connected vertex cover induces a
tree and contains all vertices of degree at least three. Therefore, a
cycle cannot consist entirely of vertices of degree at least three, since
such a cycle would be contained in the tree induced by a minimum connected
vertex cover.
If a cycle had exactly one vertex $v$ of degree two, a minimum connected
vertex cover containing $v$ would contain the entire cycle. This is also
impossible. Hence condition~\ref{cycle} holds.

\emph{\ref{cycle}$\Rightarrow$\ref{eternal-equality}.}
Finally, suppose that condition~\ref{cycle} holds. We give a strategy
starting from an arbitrary minimum connected vertex cover $C$.
If both endpoints of the attacked edge are occupied, the two endpoint guards
exchange their positions, leaving the occupied set unchanged.

Otherwise, let the attacked edge be $uv$, where $u\in C$ and $v\notin C$.
By \Cref{counting-bound}, $v$ has exactly two neighbours, say $u$ and $w$,
and both belong to $C$. Since $G[C]$ is a tree, $G[C\cup\{v\}]$ has a unique
cycle $Q$, the unique $u$ to $w$ path in $G[C]$, together with $uv$ and $vw$.
By condition~\ref{cycle}, there is a vertex $z\in V(Q)\setminus\{v\}$
with $d_G(z)=2$. Put
\[
    C'=(C\cup\{v\})\setminus\{z\}.
\]
The two neighbours of $z$ are its neighbours on $Q$. Deleting $z$ therefore
breaks the unique cycle without disconnecting the graph, so $G[C']$ is a
tree. Also, if $I=V(G)\setminus C$, then
\[
    V(G)\setminus C'=(I\setminus\{v\})\cup\{z\}.
\]
Clearly this set is independent, since $I\setminus\{v\}$ is independent, and both
neighbours of $z$ lie in $C\cup\{v\}$. Hence $C'$ is a connected vertex
cover of the same size as $C$, and is therefore minimum.

To realize this change by legal guard moves, let
\[
    u=p_0,p_1,\ldots,p_t=z
\]
be the unique $u$ to $z$ path in $G[C]$. Simultaneously move the guard on $p_0$
to $v$, and move the guard on $p_i$ to $p_{i-1}$ for each $1\leq i\leq t$.
All other guards remain fixed. If $z=u$, only the first move is needed.
Every guard moves by at most one edge, a guard crosses the attacked edge
$uv$, and the final occupied set is exactly $C'$. Thus every attack can be
defended by moving to another minimum connected vertex cover. Repeating the
same argument gives an eternal strategy from every such cover, proving
\ref{eternal-equality} and the final assertion. \qed
\end{proof}

For later applications, it is convenient to record the criterion in a form
that starts from an explicit witness cover rather than from membership in
$\mathcal F$.

\begin{corollary}\label{tree-cover-condition}
Let $G$ be a connected graph with $\delta(G)\geq 2$. Suppose that $G$ has a
connected vertex cover $C$ such that $G[C]$ is a tree and every vertex outside
$C$ has degree two. If every vertex of $G$ belongs to some minimum connected
vertex cover, then
\[
    ecvc(G)=cvc(G)=|C|=2|V(G)|-|E(G)|-1.
\]
Moreover, every minimum connected vertex cover is an eternally winning
configuration.
\end{corollary}
\begin{proof}
By \Cref{counting-bound}, $C$ is minimum and $G\in\mathcal F$.
The conclusion follows from \Cref{thm:equality-F}. \qed
\end{proof}

\subsection{Recognizing the extremal class}
\label{subsec:forest-test}

The definition of $\mathcal F$ is extremal: it refers to the value of
$cvc(G)$. To apply \Cref{thm:equality-F}, however, one would prefer a
criterion that can be read directly from the graph. The following proposition
provides exactly such a criterion by considering the subgraph induced by the
vertices of degree at least three.

\begin{proposition}\label{prop:forest-test}
Let $G$ be a connected graph with $\delta(G)\geq 2$, and put
\[
    B=\{v\in V(G):d_G(v)\geq 3\}.
\]
Then $G\in\mathcal F$ if and only if $G[B]$ is a forest.
\end{proposition}
\begin{proof}
Suppose first that $G\in\mathcal F$, and let $C$ be a minimum connected
vertex cover. By \Cref{counting-bound}, $B\subseteq C$ and $G[C]$ is a
tree. Therefore $G[B]$ is a forest.

Conversely, suppose that $G[B]$ is a forest. Write $n=|V(G)|$, $m=|E(G)|$,
and $r=m-n+1$. Extend the forest $G[B]$, together with the remaining vertices
as isolated vertices, to a spanning tree $T$ of $G$. For each edge
$e\in E(G)\setminus E(T)$, at least one endpoint has degree two, otherwise
$e$ would belong to $G[B]$ and hence to $T$. Choose one degree-two endpoint
$v_e$ of each such edge, and set
\[
    I=\{v_e:e\in E(G)\setminus E(T)\}.
\]
Since $T$ is spanning, the other edge incident with $v_e$ belongs to $T$.
Thus $v_e$ is a leaf of $T$, and $e$ is its unique incident non-tree edge.
In particular, distinct non-tree edges give distinct chosen vertices, so
$|I|=r$.

We claim that $I$ is independent. Two chosen vertices cannot be adjacent by
a tree edge, since they are leaves of a tree on at least three vertices.
If two chosen vertices were adjacent by a non-tree edge, this would be the
unique non-tree edge incident with each of them. Both would then have been
chosen for the same edge, contrary to choosing only one endpoint per edge.

Every non-tree edge has an endpoint in $I$, and the vertices of $I$ are
leaves of $T$. Consequently,
\[
    G-I=T-I
\]
is a nonempty tree. Therefore $V(G)\setminus I$ is a connected vertex cover
of size $n-r=2n-m-1$. By \Cref{counting-bound}, $G\in\mathcal F$. \qed
\end{proof}

Thus membership in $\mathcal F$ is determined entirely by the high-degree
core of the graph. In particular, any construction or graph class for which
this core is acyclic brings both the counting formula and the equality
criterion into play.

\section{Applications}
\label{sec:applications}

We now apply the criterion in two complementary settings. Full subdivision
creates the required degree-two structure explicitly, whereas minimally
$2$-connected graphs possess the relevant cycle structure by a classical
theorem of Plummer.

\subsection{Full subdivisions}
\label{sec:subdivisions}

Our first application forces the required structure by subdividing every
edge. Full subdivision separates adjacent original vertices by new
degree-two vertices. For a graph $H$, let $S(H)$
denote its \emph{full subdivision}: every edge $e=ab$ is replaced by the path
$a\,s_e\,b$, where $s_e$ is a new vertex. Thus every edge is subdivided
exactly once. The graph $S(H)$ is bipartite with parts $L=V(H)$ and
$R=\{s_e : e\in E(H)\}$, and every vertex of $R$ has degree two.

When $H$ is connected and $\delta(H)\geq 2$, the value of $ecvc(S(H))$
follows from earlier work. Theorem~4.1 of Paul and Pandey~\cite{paul2025eternal} implies
that $ecvc(G)=cvc(G)$ for every connected bipartite graph $G=(L\cup R,E)$
without pendant vertices in which $d_G(v)\leq 2$ for all $v\in R$, and $S(H)$
is such a graph. Lemma~9 of Escoffier, Gourv\`es and Monnot~\cite{escoffier2010complexity}
gives $cvc(S(H))=|L|+|L_2|-1$, where $L_2=\{v\in L : d_{S(H)}(v)\geq 2\}$.
Here $L_2=L$, so $ecvc(S(H))=cvc(S(H))=2|V(H)|-1$. The following corollary
of \Cref{thm:equality-F} recovers both facts for every bipartite graph of this
kind, not only for full subdivisions, and shows in addition that every minimum
connected vertex cover is an eternally winning configuration. In general, this
does not follow from the equality $ecvc(G)=cvc(G)$: by \Cref{lem:lemma13},
$ecvc(G^\prime)=cvc(G^\prime)$, but by \Cref{lem:type-A-loses-Gprime}, no
Type A configuration of $G^\prime$ is eternally winning.

\begin{corollary}\label{thm:full-subdivision}
Let $G$ be a connected bipartite graph with parts $L$ and $R$ such that
$\delta(G)\geq 2$ and $d_G(v)\leq 2$ for every $v\in R$. Then
\[
    ecvc(G)=cvc(G)=2|L|-1,
\]
and every minimum connected vertex cover of $G$ is an eternally winning
configuration. In particular, if $H$ is a connected graph with $n_0$ vertices
and $\delta(H)\geq 2$, then
\[
    ecvc(S(H))=cvc(S(H))=2n_0-1,
\]
and every minimum connected vertex cover of $S(H)$ is an eternally winning
configuration.
\end{corollary}

\begin{proof}
Since $\delta(G)\geq 2$, every vertex of $R$ has degree exactly two. Hence all
vertices of degree at least three lie in $L$, which is independent, so they
induce an edgeless graph, and $G\in\mathcal{F}$ by \Cref{prop:forest-test}. Every
cycle of $G$ alternates between $L$ and $R$, so it contains at least two
vertices of $R$, each of degree two. By \Cref{thm:equality-F},
$ecvc(G)=cvc(G)$ and every minimum connected vertex cover of $G$ is an
eternally winning configuration. Since $G\in\mathcal{F}$ and $|E(G)|=2|R|$,
\[
    cvc(G)=2|V(G)|-|E(G)|-1=2(|L|+|R|)-2|R|-1=2|L|-1.
\]
For the second statement, take $G=S(H)$, $L=V(H)$ and $R=\{s_e : e\in E(H)\}$;
here $\delta(S(H))\geq 2$ because $d_{S(H)}(a)=d_H(a)\geq 2$ for every
$a\in V(H)$.
\qed\end{proof}

\subsection{Revisiting the counterexample after subdivision}
\label{subsec:subdivided-counterexample}

\Cref{thm:full-subdivision} gives a transparent before-and-after comparison
for the counterexample from \Cref{sec:Counter-example}. Subdivision retains its
underlying incidence pattern but changes precisely the structural features
responsible for the failure of equality. Let $G_0$ denote the graph shown in
\Cref{fig:counterexample}. It has $32$ vertices and minimum degree two.
Its six private triangles contribute $18$ edges, the three four-cycles
contribute $12$, the pair-indexed $Z$- and $U$-attachments contribute $24$,
and the attachments of $z_0$ and $u_0$ contribute $6$. Hence
\[
    |E(G_0)|=18+12+24+6=60.
\]
By \Cref{lem:cvc}, $cvc(G_0)=19$, while \Cref{counterexamplethm} shows that
$ecvc(G_0)\neq cvc(G_0)$. The general upper bound
$ecvc(G_0)\leq cvc(G_0)+1$ therefore gives $ecvc(G_0)=20$.
On the other hand, \Cref{thm:full-subdivision} gives
\begin{equation}\label{eq:subdivided-counterexample}
    ecvc(S(G_0))=cvc(S(G_0))=63.
\end{equation}
The subdivided graph has $92$ vertices and $120$ edges, in agreement with
$2\cdot 92-120-1=63$.

This comparison identifies the structural difference between the original
graph and its subdivision. In $G_0$, the four-cycle
$a_1,q_1,c_1,r_1,a_1$ consists entirely of vertices of degree at least three.
Thus \Cref{prop:forest-test} shows that $G_0\notin\mathcal F$, despite
satisfying the vertex-membership condition. In $S(G_0)$, by contrast, the
vertices of degree at least three form an independent set, and every cycle
contains at least three new vertices of degree two. Both parts of the
positive criterion are therefore present after subdivision.

The graph $S(G_0)$ is not $2$-connected: each original support vertex remains
a cut vertex. Thus \eqref{eq:subdivided-counterexample} is not a consequence of the minimally
$2$-connected application developed next.

\subsection{Minimally $2$-connected graphs}
\label{minimally-2-connected}

Our second application is intrinsic rather than constructional. In a
minimally $2$-connected graph, the degree-two vertices already occur with the
cycle structure required by \Cref{thm:equality-F}. Recall that a graph is
\emph{minimally $2$-connected} if it is $2$-connected and deleting any edge
destroys $2$-connectivity. Plummer proved that, if $G\neq K_3$ is minimally
$2$-connected, then every cycle of $G$ contains two nonadjacent vertices of
degree two~\cite[Corollary~2a]{plummer1968minimal}; the weaker assertion
needed here is immediate for $K_3$ as well.

\begin{theorem}\label{thm:minimally-2-connected}
Let $G$ be a minimally $2$-connected graph with $n$ vertices and $m$ edges.
Then
\[
    ecvc(G)=cvc(G)=2n-m-1.
\]
Moreover, every minimum connected vertex cover of $G$ is an eternally
winning configuration.
\end{theorem}
\begin{proof}
Since $G$ is $2$-connected, $\delta(G)\geq 2$. Put
$B=\{v\in V(G):d_G(v)\geq 3\}$. Plummer's theorem implies that $G[B]$ is a
forest, since a cycle contained in $G[B]$ would have no vertex of degree
two. By \Cref{prop:forest-test}, $G\in\mathcal F$, and hence
$cvc(G)=2n-m-1$. Plummer's theorem also gives the cycle condition in
\Cref{thm:equality-F}. Applying that theorem proves
$ecvc(G)=cvc(G)$ and shows that every minimum connected vertex cover is an
eternally winning configuration. \qed
\end{proof}

\section{Conclusion}

In this paper, we studied when the eternal connected vertex cover number equals the minimum connected vertex cover number. We first showed that the condition that every vertex belongs to some minimum connected vertex cover is not sufficient for equality. We then proved a sharp lower bound for graphs with minimum degree at least two and characterized the graphs attaining it. For this class, we showed that equality holds if and only if every vertex belongs to some minimum connected vertex cover, or equivalently, every cycle contains at least two vertices of degree two. We also showed that the equality $ecvc(G)=cvc(G)$ does not, in general, imply that every minimum connected vertex cover of $G$ is an eternally winning configuration.

Finally, we applied these results to minimally $2$-connected graphs, obtaining an exact formula for their eternal connected vertex cover number, and recovered the known formula for full subdivisions~\cite{paul2025eternal,escoffier2010complexity}. In both cases, every minimum connected vertex cover is an eternally winning configuration.

\section*{Declaration of generative AI use}
During the preparation of this manuscript, the authors used ChatGPT to assist in developing the counterexample in \Cref{sec:Counter-example,sec:modified-example}, identifying minimally $2$-connected graphs as a class satisfying $ecvc(G)=cvc(G)$, improving the language and organization of the manuscript. The authors reviewed and revised the AI-assisted material and take full responsibility for the mathematical correctness and final content of the manuscript.

\section*{Acknowledgments}
Rajat Adak acknowledges support from the TCS Research Scholar Program. Saraswati Girish Nanoti thanks ANRF India for supporting this work through NPDF fellowship research grant SP/ANRF 25-0599.06.

\bibliographystyle{plain}
\bibliography{references}

@inproceedings{fujito2020eternal,
  title={Eternal connected vertex cover problem},
  author={Fujito, Toshihiro and Nakamura, Tomoya},
  booktitle={International Conference on Theory and Applications of Models of Computation},
  pages={181--192},
  year={2020},
  organization={Springer}
}

@article{plummer1968minimal,
  title={On minimal blocks},
  author={Plummer, Michael D},
  journal={Transactions of the American Mathematical Society},
  volume={134},
  number={1},
  pages={85--94},
  year={1968},
  publisher={JSTOR}
}

@article{klostermeyer2009edge,
  title={Edge protection in graphs},
  author={Klostermeyer, William and Mynhardt, Kieka},
  journal={Australasian Journal of Combinatorics},
  volume={45},
  pages={235--250},
  year={2009}
}

@article{babu2022graphs,
  title={On graphs whose eternal vertex cover number and vertex cover number coincide},
  author={Babu, Jasine and Chandran, L Sunil and Francis, Mathew and Prabhakaran, Veena and Rajendraprasad, Deepak and Warrier, Nandini J},
  journal={Discrete Applied Mathematics},
  volume={319},
  pages={171--182},
  year={2022},
  publisher={Elsevier}
}

@inproceedings{misra2025characterization,
  title={A Characterization of Spartan Graphs and New Lower Bounds for Eternal Vertex Cover},
  author={Misra, Neeldhara and Nanoti, Saraswati Girish},
  booktitle={45th IARCS Annual Conference on Foundations of Software Technology and Theoretical Computer Science (FSTTCS 2025)},
  pages={45--1},
  year={2025},
  organization={Schloss Dagstuhl--Leibniz-Zentrum f{\"u}r Informatik}
}

@inproceedings{paul2023some,
  title={Some algorithmic results for eternal vertex cover problem in graphs},
  author={Paul, Kaustav and Pandey, Arti},
  booktitle={International Conference and Workshops on Algorithms and Computation},
  pages={242--253},
  year={2023},
  organization={Springer}
}

@article{paul2025eternal,
  title={Eternal connected vertex cover problem in graphs: Complexity and algorithms},
  author={Paul, Kaustav and Pandey, Arti},
  journal={Theoretical Computer Science},
  pages={115509},
  year={2025},
  publisher={Elsevier}
}

@article{escoffier2010complexity,
  title={Complexity and approximation results for the connected vertex cover problem in graphs and hypergraphs},
  author={Escoffier, Bruno and {Gourv{\`e}s}, Laurent and Monnot, {J{\'e}r{\^o}me}},
  journal={Journal of Discrete Algorithms},
  volume={8},
  number={1},
  pages={36--49},
  year={2010},
  publisher={Elsevier}
}

\end{document}